\documentclass[a4paper,11pt,reqno]{amsart}
\usepackage{amsmath,amssymb,amsthm}
\usepackage[dvipsnames]{xcolor}
\usepackage[colorlinks=true,citecolor=RoyalBlue,linkcolor=red,breaklinks=true]{hyperref}
\usepackage{pictures}
\usepackage{graphicx}
\usepackage{mathdots} % for anti diagonal dots
\usepackage[margin=2.5cm]{geometry}
\usepackage{ytableau}

\newtheorem{thm}{Theorem}[section]
\newtheorem{lem}[thm]{Lemma}
\newtheorem{prop}[thm]{Proposition}

\newtheorem{rem}[thm]{Remark}

\theoremstyle{definition}

\newcommand{\x}{\mathbf{x}}
\newcommand{\la}{\lambda}
\newcommand{\y}{\mathbf{y}}
\newcommand{\A}{\mathcal{A}}

\DeclareMathOperator{\ASM}{ASM}
\DeclareMathOperator{\TSPP}{TSPP}
\DeclareMathOperator{\CSSPP}{CSSPP}

\DeclareMathOperator{\sgn}{sgn}
\DeclareMathOperator{\diag}{diag}
\DeclareMathOperator{\inv}{inv}

\DeclareMathOperator{\asym}{\mathbf{ASym}}

\newcommand{\red}[1]{{\color{red}{#1}}}

\newcommand{\cyan}[1]{{\color{cyan}{#1}}}

\newmuskip\pFqmuskip

\newcommand*\pFq[6][8]{%
	\begingroup % only local assignments
	\pFqmuskip=#1mu\relax
	\mathchardef\normalcomma=\mathcode`,
	\mathcode`\,=\string"8000
	\begingroup\lccode`\~=`\,
	\lowercase{\endgroup\let~}\pFqcomma
	{}_{#2}F_{#3}{\left[\left.\genfrac..{0pt}{}{#4}{#5}\right|#6\right]}%
	\endgroup
}
\newcommand{\pFqcomma}{{\normalcomma}\mskip\pFqmuskip}

\numberwithin{equation}{section}

\title[Alternating sign matrices and totally symmetric plane partitions]{Alternating sign matrices and totally symmetric plane partitions}

\author[F. Aigner]{Florian Aigner}
\address{Fakult\"at f\"ur Mathematik, Universit\"at Wien, Austria}
\email{florian.aigner@univie.ac.at}
\urladdr{https://homepage.univie.ac.at/florian.aigner/}

\author[I. Fischer]{Ilse Fischer}
\address{Fakult\"at f\"ur Mathematik, Universit\"at Wien, Austria}
\email{ilse.fischer@univie.ac.at}
\urladdr{https://www.mat.univie.ac.at/$\sim$ifischer/}

\thanks{The first author acknowledges support from the Austrian Science Foundation FWF: J 4387 and SFB grant F50 and by the project ‘‘Austria/France Scientific \& Technological Cooperation’’ (BMWFW Project No. FR 10/2018 and PHC Amadeus 2018 Project No. 39444WJ). The second author acknowledges support from the Austrian Science Foundation FWF: SFB grant F50 and grant P34931.}


\begin{document}

\begin{abstract}
We introduce a new family $\mathcal{A}_{n,k}$ of Schur positive symmetric functions, which are defined as sums over totally symmetric plane partitions. In the first part, we show that, for $k=1$, this family is equal to a multivariate generating function involving $n+3$ variables of objects that extend alternating sign matrices (ASMs), which have recently been introduced by the authors. This establishes a new connection between ASMs and a class of plane partitions, thereby complementing the fact that ASMs are equinumerous with totally symmetric self-complementary plane partitions as well as with descending plane partitions. The proof is based on a new antisymmetrizer-to-determinant formula for which we also provide a bijective proof.
In the second part, we relate three specialisation of $\mathcal{A}_{n,k}$ to a weighted enumeration of certain well-known classes of column strict shifted plane partitions that generalise descending plane partitions.
\end{abstract}

\maketitle

%%%%%%%%%%%%%%%%%%%%%%%%%%%%%%%%%%%%%%%%%%%%%%%%%%%%%%%%%%%%%%%%%%%%%%%%%%%%%%%%%%%%%%%%%%%%%%%%%%%%%%%%%%%%
%%%%%%%%%%%%%%%%%%%%%%%%%%%%%%%%%%%%%%%%%%%%%%%%%%%%%%%%%%%%%%%%%%%%%%%%%%%%%%%%%%%%%%%%%%%%%%%%%%%%%%%%%%%%
%                                            1) Introduction
%%%%%%%%%%%%%%%%%%%%%%%%%%%%%%%%%%%%%%%%%%%%%%%%%%%%%%%%%%%%%%%%%%%%%%%%%%%%%%%%%%%%%%%%%%%%%%%%%%%%%%%%%%%%
%%%%%%%%%%%%%%%%%%%%%%%%%%%%%%%%%%%%%%%%%%%%%%%%%%%%%%%%%%%%%%%%%%%%%%%%%%%%%%%%%%%%%%%%%%%%%%%%%%%%%%%%%%%%

\section{Introduction}
\emph{Plane partitions} were first studied by MacMahon \cite{MacMahon97} at the end of the 19th century, however found broader interest in the combinatorial community starting in the second half of the last century. \emph{Alternating sign matrices} (ASMs) on the other hand were introduced by Robbins and Rumsey \cite{RobbinsRumsey86} in the early 1980s. Together with Mills \cite{MillsRobbinsRumsey82}, they conjectured that the number of $n \times n$ ASMs is given by $\prod_{i=0}^{n-1} \frac{(3i+1)!}{(n+i)!}$. Stanley then pointed out that these numbers had appeared before in the work of Andrews \cite{Andrews79} as the enumeration formula for a certain class of plane partitions, called \emph{descending plane partitions} (DPPs). Soon after that Mills, Robbins and Rumsey \cite{MillsRobbinsRumsey86} observed (conjecturally) that this formula also counts another class of plane partitions, namely \emph{totally symmetric self-complementary plane partitions} (TSSCPPs). Although these conjectures have all been proved since then, see among others \cite{Andrews94,Zeilberger96}, it is mostly agreed that there is no good combinatorial understanding of this relation between ASMs and certain classes of plane partitions since we lack transparent combinatorial proofs of these results. However, Konvalinka and the second author \cite{cube,FischerKonvalinka20plus} have recently established complicated bijective proofs (involving a generalisation of the involution principle) for an identity that implies the equinumerosity of ASMs and DPPs as well as for the product formula.

One purpose of this paper is to relate ASMs to yet another class of plane partitions, namely \emph{totally symmetric plane partitions} (TSPPs), in a new way. This relation is via a certain Schur polynomial expansion. 
Other known relations between ASMs and TSPPs are the fact that the number of symmetric plane partitions inside an $(n,n,n-1)$-box is the product of the number of TSPPs inside an $(n-1,n-1,n-1)$-box and the number of ASMs of size $n$, see \cite{Fischer05}, and via posets, see \cite[Section 8]{Striker11}.

The following symmetric functions are studied in this paper
\[
\A_{n,k}(r,u,v,w; \x) := \sum_{T \in \TSPP_{n-1}} \omega(T) s_{\pi_k(T)}(\x),
\]
where the sum is over totally symmetric plane partitions inside an $(n-1,n-1,n-1)$-box, $\pi_k(T)$ is a slight modification of the diagonal of $T$ and $\omega(T)$ is a monomial in $r,u,v,w$ that depends on the parameters in the Frobenius notation of $\pi_0(T)$.
All notations in the introduction are defined in the following sections. For $n=3$, the function $\A_{3,k}$ is a sum over all TSPPs inside a $(2,2,2)$-box, see Figure \ref{fig: Example for Ank}, and it is equal to
\begin{multline*}
\A_{3,k}(r,u,v,w,\x)= v^3 + r  u v^2 s_{(1^{k+1})}(\x) + r  uv w s_{(1^{k+2})}(\x) + r  u^2 v  s_{(2,1^{k+1})}(\x) + r^2u^3 s_{(2^{k+2})}(\x).
\end{multline*}

\begin{figure}[h]
\begin{center}
	\begin{tikzpicture}
	\begin{scope}[scale=0.42]
		\node at (-1.5,0) {$T$:};
		\node at (2,0) {$\emptyset$};
		\begin{scope}[xshift=6cm]
			\PlanePartition{{1,0},{0,0}}
		\end{scope}
		\begin{scope}[xshift=12cm]
			\PlanePartition{{2,1},{1,0}}
		\end{scope}
		\begin{scope}[xshift=18cm]
			\PlanePartition{{2,2},{2,1}}
		\end{scope}
		\begin{scope}[xshift=24cm]
			\PlanePartition{{2,2},{2,2}}
		\end{scope}
	\end{scope}
	\begin{scope}[scale=0.42, yshift=-6cm,]
		\node at (-1.5,2) {$\pi_k(T)$:};
		\node at (2,2) {$\emptyset$};
		\node at (6,2) {$(1^{k+1})$};
		\node at (12,2) {$(1^{k+2})$};
		\node at (18,2) {$(2,1^{k+1})$};
		\node at (24,2) {$(2^{k+2})$};
	\end{scope}
	\begin{scope}[scale=0.42, yshift=-7cm]
		\node at (-1.5,0) {$\omega(T)$:};
		\node at (2,0) {$ v^3$};
		\node at (6,0) {$r  u v^2$};
		\node at (12,0) {$r  u v w$};
		\node at (18,0) {$r  u^2 v$};
		\node at (24,0) {$r^2 u^3$};
	\end{scope}
	\end{tikzpicture}
\end{center}
\caption{\label{fig: Example for Ank} Totally symmetric plane partitions inside a $(2,2,2)$-box, their associated weight $\omega(T)$ and the partitions $\pi_k(T)$.}
\end{figure}

Our first main result states that in the special case $k=1$, the above functions give the Schur polynomial expansion of a weighted generating function for ASMs, which has recently been introduced by the authors in \cite{AignerFischer2106.11568}.

\begin{thm}
\label{thm: main thm 1}
For all positive integers $n$, the weighted generating function for ASMs with respect to the weight $\omega_A$ is equal to
\begin{equation}
\label{eq: main thm 1}
\sum_{A \in \ASM_n} \omega_A(u,v,w;\x) = \A_{n,1}(1, u,v,w;\x).
\end{equation}
\end{thm}

Our proof of this result is (mostly) non-combinatorially, and thus it adds another problem to the growing zoo of (obviously challenging) bijective proof problems related to ASMs and plane partitions. More specifically, it suggests that there is a bijection between the down-arrowed monotone triangles\footnote{These are certain decorated monotone triangles. Monotone triangle are in easy bijective correspondence with ASMs.} from \cite{AignerFischer2106.11568}, and pairs of totally symmetric plane partitions and semistandard Young tableaux. Moreover, \eqref{eq: main thm 1} involves $n+4$ parameters, and, therefore, we even have a considerable number of equidistributed statistics that could help in finding such a bijection.

\medskip

In the second part of our paper, we consider the case of general $k$ and connect the family $\A_{n,k}$ of symmetric functions to another family of plane partitions, namely \emph{column strict shifted plane partitions} (CSSPPs) of class $k$. CSSPPs of class $k$ form a family of plane partitions, generalising cyclically symmetric plane partitions (CSPPs) and DPPs in the sense that they are in bijection to CSPPs for $k=0$ and to DPPs for $k=2$. 
Let $\CSSPP_{n,k}(r,t)$ denote a certain generating function of CSSPPs of class $k$ with at most $n$ entries in the first row; for the definitions see Section \ref{sec: CSSPPs}. Then our second main theorem states the following.

\begin{thm}
\label{thm: main thm 2}
Let $n$ be a positive integer and let $\mathbf{1}=(1, \ldots, 1)$. Then,
\begin{align}
\label{eq: Ank CSSPP 1}
\A_{n+1,0}(r,1,1,t;\mathbf{1})&=\CSSPP_{n,0}(r,t+2),\\
\label{eq: Ank CSSPP 2}
\A_{n+1,k}(r,1,1,-1;\mathbf{1})&=\CSSPP_{n,2k}(r,1),\\
\label{eq: Ank CSSPP 3}
\A_{n+1,k}(r,1,1,0;\mathbf{1})&=\CSSPP_{n,k}(r,2).
\end{align}
\end{thm}
For $k=1$, the identity \eqref{eq: Ank CSSPP 2} is closely related a special case of \cite[Theorem 2.5]{AignerFischer2106.11568}. The choice of the parameters $(u,v,w)=(1,1,-1)$ in 
\eqref{eq: Ank CSSPP 2} corresponds to the straight enumeration of ASMs, while the choice $(u,v,w)=(1,1,0)$ in 
\eqref{eq: Ank CSSPP 3} corresponds to the $2$-enumeration of ASMs, which is related to the straight enumeration of the Aztec diamond.
\bigskip

The structure of the paper is as follows. In Section~\ref{sec: TSPPs}, we  recall some basics of plane partitions and introduce the family $\A_{n,k}$ of symmetric functions in detail. In Section~\ref{sec: ASM gen fct}, we provide the definition of the symmetric generating function for ASMs and relate the weight for ASMs to the six-vertex model. Section~\ref{sec: antisym to det} contains Lemma  \ref{lem: general}, which allows us to express certain antisymmetrisers as determinants. We provide two proofs of this lemma: one is using linear algebra, and the second is combinatorial in nature and based on directed graphs. Section~\ref{sec: proof of mt 1} contains the proof of Theorem~\ref{thm: main thm 1}. In Section~\ref{sec: CSSPPs}, we recall CSSPPs and provide in Lemma~\ref{lem: Ank via det} a determinantal description of $\A_{n,k}$ closely related to the Giambelli identity for Schur functions. The proof of Theorem~\ref{thm: main thm 2} is presented in Section~\ref{sec: proof of mt 2}.\bigskip

An extended abstract containing parts of Section~\ref{sec: TSPPs}--\ref{sec: proof of mt 1} was published in the proceedings of FPSAC 2020 \cite{AignerFischerKonvalinkaNadeauTewari20}.

%%%%%%%%%%%%%%%%%%%%%%%%%%%%%%%%%%%%%%%%%%%%%%%%%%%%%%%%%%%%%%%%%%%%%%%%%%%%%%%%%%%%%%%%%%%%%%%%%%%%%%%%%%%%
%%%%%%%%%%%%%%%%%%%%%%%%%%%%%%%%%%%%%%%%%%%%%%%%%%%%%%%%%%%%%%%%%%%%%%%%%%%%%%%%%%%%%%%%%%%%%%%%%%%%%%%%%%%%
%                                            2) TSPPs and sym fct
%%%%%%%%%%%%%%%%%%%%%%%%%%%%%%%%%%%%%%%%%%%%%%%%%%%%%%%%%%%%%%%%%%%%%%%%%%%%%%%%%%%%%%%%%%%%%%%%%%%%%%%%%%%%
%%%%%%%%%%%%%%%%%%%%%%%%%%%%%%%%%%%%%%%%%%%%%%%%%%%%%%%%%%%%%%%%%%%%%%%%%%%%%%%%%%%%%%%%%%%%%%%%%%%%%%%%%%%%

\section{A family of symmetric functions related to TSPPs}
\label{sec: TSPPs}

A \emph{partition} $\la=(\la_1,\ldots,\la_n)$ is a weakly decreasing sequence of non-negative integers (we deviate from the more usual definition where parts have to be positive). We identify a partition $\la$ with its \emph{Young diagram}, which is a collection of left-justified boxes with $\la_i$ boxes in the $i$-th row from bottom (using French notation). The \emph{conjugate} $\la^\prime$ of $\la$ is the partition obtained by reflecting the Young diagram along the $y=x$ axis, i.e., $\la^\prime_i=|\{j: \la_j \geq i\}|$. The \emph{Durfee square} of a partition $\la$ is the largest square which fits into the Young diagram. The \emph{Frobenius notation} of a partition $\la$ is $(\la_1-1,\ldots,\la_l-l|\la_1^\prime-1,\ldots,\la_l^\prime-l)$, where $l= \max_i\{ \la_i \geq i \}$ is the length of the Durfee square of $\la$.
\bigskip

Let $k$ be a non-negative integer. A \emph{$k$-tall partition}\footnote{For $k=0$, these objects were defined in \cite[Ex 6.16(bb), p.223]{Stanley99} without a name, and for $k=1$ in \cite{AignerFischerKonvalinkaNadeauTewari20} as modified balanced partitions.} $\la$ of size $n$ is a partition $\la=(\la_1,\ldots,\la_{n+k-1})$ with $\la_1\leq n-1$ that satisfies $\la_i +k \leq \la_i^\prime$ whenever $\la_i \geq i$. See Figure \ref{fig: mod bal part n=3} for an example.
 If $\la$ has Frobenius notation $(a_1,\ldots,a_l|b_1+k,\ldots,b_l+k)$, then $\la$ is a $k$-tall partition iff $a_i \leq b_i$ for all $1 \leq i \leq l$. Let $N$ denote a unit north-step and $E$ a unit east-step. The map
\[
(a_1,\ldots,a_l|b_1+k,\ldots,b_l+k) \mapsto
N^{b_l+1}E^{a_l+1}N^{b_{l-1}-b_l}E^{a_{l-1}-a_l} \cdots N^{b_1-b_2}E^{a_1-a_2}N^{n-b_1-1}E^{n-a_1-1}
\]
and $(|) \mapsto N^nE^n$ is a bijection from $k$-tall partitions of size $n$ to Dyck paths of length $2n$.
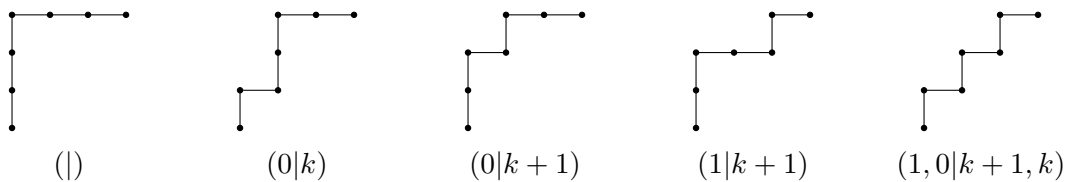
\begin{figure}[h]
\begin{center}
\begin{tikzpicture}
	\draw (0,0)--(0,1.5) -- (1.5,1.5);
	\draw[fill] (0,0) circle (1pt);
	\draw[fill] (0,0.5) circle (1pt);
	\draw[fill] (0,1) circle (1pt);
	\draw[fill] (0,1.5) circle (1pt);
	\draw[fill] (0.5,1.5) circle (1pt);
	\draw[fill] (1,1.5) circle (1pt);
	\draw[fill] (1.5,1.5) circle (1pt);
	\node at (.75,-.5) {$(|)$};
\begin{scope}[xshift=3cm]
	\draw (0,0)--(0,.5) -- (.5,.5) -- (0.5,1.5) -- (1.5,1.5);
	\draw[fill] (0,0) circle (1pt);
	\draw[fill] (0,0.5) circle (1pt);
	\draw[fill] (0.5,.5) circle (1pt);
	\draw[fill] (0.5,1) circle (1pt);
	\draw[fill] (0.5,1.5) circle (1pt);
	\draw[fill] (1,1.5) circle (1pt);
	\draw[fill] (1.5,1.5) circle (1pt);
	\node at (.75,-.5) {$(0|k)$};	
\end{scope}
\begin{scope}[xshift=6cm]
	\draw (0,0)--(0,1) -- (0.5,1) -- (0.5,1.5) -- (1.5,1.5);
	\draw[fill] (0,0) circle (1pt);
	\draw[fill] (0,0.5) circle (1pt);
	\draw[fill] (0,1) circle (1pt);
	\draw[fill] (0.5,1) circle (1pt);
	\draw[fill] (0.5,1.5) circle (1pt);
	\draw[fill] (1,1.5) circle (1pt);
	\draw[fill] (1.5,1.5) circle (1pt);
	\node at (.75,-.5) {$(0|k+1)$};
\end{scope}
\begin{scope}[xshift=9cm]
	\draw (0,0)--(0,1) -- (1,1) -- (1,1.5) -- (1.5,1.5);
	\draw[fill] (0,0) circle (1pt);
	\draw[fill] (0,0.5) circle (1pt);
	\draw[fill] (0,1) circle (1pt);
	\draw[fill] (0.5,1) circle (1pt);
	\draw[fill] (1,1) circle (1pt);
	\draw[fill] (1,1.5) circle (1pt);
	\draw[fill] (1.5,1.5) circle (1pt);
	\node at (.75,-.5) {$(1|k+1)$};
\end{scope}
\begin{scope}[xshift=12cm]
	\draw (0,0)--(0,.5) -- (.5,.5) -- (.5,1) -- (1,1) -- (1,1.5) -- (1.5,1.5);
	\draw[fill] (0,0) circle (1pt);
	\draw[fill] (0,0.5) circle (1pt);
	\draw[fill] (.5,.5) circle (1pt);
	\draw[fill] (0.5,1) circle (1pt);
	\draw[fill] (1,1) circle (1pt);
	\draw[fill] (1,1.5) circle (1pt);
	\draw[fill] (1.5,1.5) circle (1pt);
	\node at (.75,-.5) {$(1,0|k+1,k)$};
\end{scope}
\end{tikzpicture}
\caption{\label{fig: mod bal part n=3} All $k$-tall partitions of size $3$ in Frobenius notation together with their associated Dyck paths.}
\end{center}
\end{figure}
\bigskip

A \emph{plane partition} $\pi$ inside an $(a,b,c)$-box is an array $(\pi_{i,j})_{1 \leq i \leq a, 1 \leq j \leq b}$ of non-negative integers less than or equal to $c$, with weakly decreasing rows and columns, i.e., $\pi_{i,j} \geq \pi_{i,j+1}$ and  $\pi_{i,j} \geq \pi_{i+1,j}$. We can visualise a plane partition $\pi$ as stacks of unit cubes by putting $\pi_{i,j}$ cubes at position $(i,j)$, see Figure \ref{fig: ex PP}. The visualisation allows an equivalent definition of plane partitions: A plane partition $\pi$ inside an $(a,b,c)$-box is a subset of $[a]\times[b] \times[c]$, where $[n]=\{1,\ldots,n\}$, such that $(i,j,k) \in \pi$ implies $(i^\prime,j^\prime,k^\prime) \in \pi$ for all $i^\prime \leq i, j^\prime \leq j, k^\prime \leq k$.
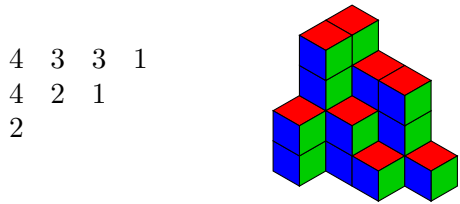
\begin{figure}
\begin{center}
\begin{tikzpicture}
\begin{scope}
	\node at (0,0) {$\begin{array}{cccc}
	4 & 3 & 3 & 1\\
	4 & 2 & 1 \\
	2
	\end{array}$};
	\end{scope}
	\begin{scope}[scale=0.4, xshift=9cm, yshift=-1cm]
		\PlanePartition{{4,3,3,1},{4,2,1,0},{2,0,0,0}}
	\end{scope}
\end{tikzpicture}
\caption{\label{fig: ex PP} A plane partition inside a $(3,4,4)$-box, with $0$ entries omitted, and its graphical representation as stacks of cubes.}
\end{center}
\end{figure}
A plane partition is called \emph{totally symmetric} if for every $(i,j,k) \in \pi$, all permutations of $(i,j,k)$ are also elements of $\pi$. We denote by $\TSPP_n$ the set of totally symmetric plane partitions (TSPPs) inside an $(n,n,n)$-box. Let $T=(T_{i,j})_{1\leq i,j \leq n}$ be a totally symmetric plane partition, $\diag(T)=(T_{i,i})_{1 \leq i \leq n}^\prime$ its diagonal (note that we conjugate) and $(a_1,\ldots,a_l|b_1,\ldots,b_l)$ the Frobenius notation of $\diag(T)$. The partition $\diag(T)$ describes the shape which is obtained by intersecting the visualisation of $T$ as stacks of cubes with the $x=y$ plane. See Figure \ref{fig:LGV} right for an example. We associate with $T$ the partition $\pi_k(T)=(a_1,\ldots,a_l|b_1+k,\ldots,b_l+k)$ of size $n+1$. As a consequence of the next proposition, we obtain that $\pi_k(T)$ is a $k$-tall partition of size $n+1$.

\begin{prop}
\label{prop: TSSPP refinement}
Let $\lambda=(a_1,\ldots,a_l|b_1+k,\ldots,b_l+k)$ be a $k$-tall partition. The number of totally symmetric plane partitions $T$ with $\pi_k(T)=\lambda$ is given by
\[
\det_{1 \leq i,j \leq l}\left( \binom{b_i}{a_j}\right).
\]
\end{prop}
\begin{proof}
This is a classical application of the Lindstr\"om-Gessel-Viennot theorem~\cite{GesselViennot85,Lindstroem73}, see also \cite{Stembridge95}. 
We sketch the proof on the example in Figure~\ref{fig:LGV}.  

TSPPs of order $n$ clearly correspond to lozenge tilings of a regular hexagon with side lengths $n$ that are symmetric with respect to the vertical symmetry axis as well as rotation of $120^\circ$. By this symmetry, it suffices to know a sixth of the lozenge tiling. In our example, we choose the sixth that is in the wedge of the red dotted rays.
\begin{figure}[h]
\centering
\begin{tikzpicture}
\begin{scope}[scale=0.5]
\PlanePartitionWhite{{4,4,4,3},{4,3,2,1},{4,2,1,1},{3,1,1,0}}
\RightBoundary{0}{-4}{4}{3}
\LeftBoundary{4}{0}{4}{3}
\TopBoundary{4}{-4}{0}{3}
\draw [color=green, line width=1.5pt] (0,1) -- ({1*cos(30)},{1+1*sin(30)}) -- (0,2) -- ({-1*cos(30)},{1+1*sin(30)}) -- (0,1);
\draw [color=green, line width=1.5pt] (0,3) -- ({1*cos(30)},{3+1*sin(30)}) -- (0,4) -- ({-1*cos(30)},{3+1*sin(30)}) -- (0,3);
\draw [color=green, line width=1.5pt] (0,-1) -- (0,0);
\draw [color=green, line width=1.5pt] (0,-3) -- (0,-2);
\draw [color=green, line width=1.5pt] (0,0) -- ({-1*cos(30)},{1*sin(30)});
\draw [color=green, line width=1.5pt] ({-3*cos(30)},{3*sin(30)}) -- ({-2*cos(30)},{2*sin(30)});
\draw [dashed, red] (0,0) -- (0,5);
\draw [dashed, red] (0,0) -- ({-5.25*cos(30)},{5.25*sin(30)});
\draw [dashed, blue, thick] ({1/2*cos(30)},{1+3/2*sin(30)}) -- ({-1/2*cos(30)},{1+1/2*sin(30)})  -- ({-1/2*cos(30)},{1/2*sin(30)});
\draw [dashed, blue, thick] ({1/2*cos(30)},{1+11/2*sin(30)}) -- ({-5/2*cos(30)},{1+5/2*sin(30)}) -- ({-5/2*cos(30)},{5/2*sin(30)});

\begin{scope}[xshift=8cm, yshift=-3.5cm, scale=1.1]
\Square{1}{4}
\Square{1}{3}
\Square{1}{2}
\Square{1}{1}
\Square{2}{3}
\Square{2}{2}
\Square{2}{1}
\Square{3}{1}
\draw [color=green, line width=1.5pt] (1,5) -- (2,5);
\draw [color=green, line width=1.5pt] (2,4) -- (3,4);
\draw [color=green, line width=1.5pt] (3,3) -- (3,2);
\draw [color=green, line width=1.5pt] (4,2) -- (4,1);
\end{scope}
\end{scope}
\end{tikzpicture}
\caption{\label{fig:LGV} Running example in the proof of Proposition~\ref{prop: TSSPP refinement} (left) and its diagonal $\diag(T)$ (right), with the corresponding edges also coloured in green.}
\end{figure}
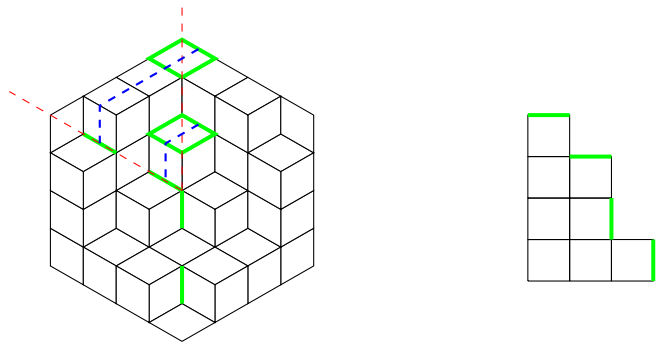 

\noindent
Now observe that the positions of the horizontal lozenges in the upper half of the vertical symmetry axis are prescribed by the $b_i$'s, while the positions of the vertical segments in the lower part of the vertical symmetry axis are prescribed by the $a_i$'s. Both are indicated in green in Figure~\ref{fig:LGV}. By the cyclic symmetry, these green segments have corresponding segments on the red dotted ray that is not contained on the vertical symmetry axis, again indicated in green in the figure. Now the lozenge tiling is determined by the family of non-intersecting lattice paths that connect these segments with the horizontal lozenges in the upper half of the vertical symmetry axis, indicated in blue in the figure.
\end{proof}

Let $\lambda=(\lambda_1,\ldots,\lambda_n)  \subseteq (n^n)$ be a partition with Frobenius notation $(a_1,\ldots,a_l|b_1,\ldots,b_l)$ and define $\lambda^c=(\lambda_i^c)_{1 \leq i \leq n}$ by $\lambda_{n+1-i}^c=n-\lambda_i$. Then $\lambda^c$ is the complement of $\lambda$ inside the partition $(n^n)$ in the sense that we can fill a square of side length $n$ by the Young diagrams of $\la$ and $\lambda^c$ without overlap, see Figure \ref{fig: ex complement of partition} for an example.
\begin{figure}[h]
\begin{center}
\begin{tikzpicture}
\begin{scope}[scale=.65]
\draw[fill = blue!30!white] (0,0) rectangle (4,4);
\draw[fill = red!40!white] (4,4) rectangle (6,6);
\foreach \x in {0,1,...,6}{
	\draw[gray, line width = 0.5pt] (0,\x) -- (6,\x);
	\draw[gray, line width = 0.5pt] (\x,0) -- (\x,6);
}
\draw[line width = 1.5pt] (0,0) -- (6,0) -- (6,2) -- (5,2) -- (5,4) -- (3,4) -- (3,5) -- (1,5) -- (1,6) -- (0,6)  -- (0,0);
\draw[line width = 1.5pt] (1,6) -- (6,6) -- (6,2);
\end{scope}
\end{tikzpicture}
\caption{\label{fig: ex complement of partition} The partition $\lambda=(6,6,5,5,3,1)$ in the bottom left and its complement $\lambda^c = (5,3,1,1)$ in the top right. Their corresponding Durfee squares are coloured in blue or red respectively.}
\end{center}
\end{figure}
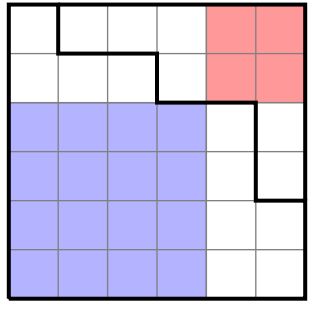
  Let $( a_1^c,\ldots,a_L^c|b_1^c,\ldots,b_L^c)$ be the Frobenius notation of $\lambda^c$. Every box of the $x=y$ diagonal of the $n \times n$ square is either in the Durfee square of $\lambda$ or of $ \lambda^c$. Hence we have $l+L=n$. Using induction on $n$, one can show %
\begin{equation}
\label{eq: Frobenius coef of complement}
\{a_1,\ldots,a_l, b_1^c,\ldots, b_L^c \}=\{ a_1^c,\ldots,a_L^c,b_1,\ldots,b_l\}=\{0,\ldots,n-1\}.
\end{equation}
For a totally symmetric plane partition $T=(T_{i,j})_{1 \leq i,j \leq n}$ inside an $(n,n,n)$-box denote by $T^c$ the complement of $T$, defined by
\[
T^c = (n-T_{n+1-i,n+1-j})_{1 \leq i,j \leq n}.
\]
The map $T \mapsto T^c$ is an involution on  totally symmetric plane partitions inside an $(n,n,n)$-box which satisfies $\pi_0(T^c) = \pi_0(T)^c$. Together with Proposition \ref{prop: TSSPP refinement}, this implies
\begin{equation}
\label{eq: det of complement}
\det_{1 \leq i,j \leq l}\left( \binom{b_i}{a_j}\right) =
\det_{1 \leq i,j \leq n-l}\left( \binom{b_i^c}{a_j^c}\right).
\end{equation}

\bigskip

Denote by $\A_{n,k}(r, u,v,w;\x)$ the symmetric polynomials in $\x=(x_1,\ldots,x_{n+k-1})$ defined by
\begin{equation}
\label{eq: def of A}
\A_{n,k}(r, u,v,w;\x)=\sum_{T \in \TSPP_{n-1}}\omega(T) s_{\pi_k(T)}(\x),
\end{equation}
where 
\[
\omega(T)=r^l u^{\sum\limits_{i=1}^l (a_i+1)}  v^{\binom{n}{2}-\sum\limits_{i=1}^l (b_i+1)} w^{\sum\limits_{i=1}^l (b_i-a_i)},
\]
if $\diag(T)$ has Frobenius notation $(a_1,\ldots,a_l|b_1,\ldots,b_l)$. We list this family of symmetric functions for $n\leq 4$.

\begin{align*}
\A_{1,k} (r, u,v,w,t;\x) &=1, \\
\A_{2,k} (r, u,v,w,t;\x) &= v + r u \,s_{(0|k)}(\x), \\
\A_{3,k} (r, u,v,w,t;\x) &=  v^3 + r  u v^2  \, s_{(0|k)}(\x) + 
	r  u v w  \,s_{(0|k+1)}(\x)+ r   u^2 v \,s_{(1|1+k)}(\x) \\
	& \quad +r^2 u^3 \,s_{(1,0|1+k,k)}(\x), \\ 
\A_{4,k} (r, u,v,w,t;\x) &=   v^6
	+r  u v^5  \,s_{(0|k)}(\x)
	+r  u v^4 w \,s_{(0|k+1)}(\x)
	+r  u^2 v^4 \,s_{(1|k+1)}(\x) \\ 
	& \quad +r  u v^3 w^2 \,s_{(0|k+2)}(\x) 
	+2r  u^2 v^3 w \,s_{(1|k+2)}(\x) 
	+r  u^3 v^3 \,s_{(2|k+2)}(\x)  \\
	& \quad +r^2  u^3 v^3 \,s_{(1,0|k+1,k)}(\x) 
	+2r^2  u^3 v^2 w  \,s_{(1,0|k+2,k)}(\x) 
	+r^2  u^4 v^2 \,s_{(2,0|k+2,k)}(\x)  \\ 
	& \quad +r^2  u^3 v w^2 \,s_{(1,0|k+2,k+1)}(\x)
	+r^2  u^4 v w \,s_{(2,0|k+2,k+1)}(\x) \\
	& \quad+r^2  u^5 v \,s_{(2,1|k+2,k+1)}(\x) 
	 +r^3 u^6 \,s_{(2,1,0|k+2,k+1,k)}(\x).
	%\label{eq: ex of A}
\end{align*}

%%%%%%%%%%%%%%%%%%%%%%%%%%%%%%%%%%%%%%%%%%%%%%%%%%%%%%%%%%%%%%%%%%%%%%%%%%%%%%%%%%%%%%%%%%%%%%%%%%%%%%%%%%%%
%%%%%%%%%%%%%%%%%%%%%%%%%%%%%%%%%%%%%%%%%%%%%%%%%%%%%%%%%%%%%%%%%%%%%%%%%%%%%%%%%%%%%%%%%%%%%%%%%%%%%%%%%%%%
%                                            3) Sym gen for ASMs
%%%%%%%%%%%%%%%%%%%%%%%%%%%%%%%%%%%%%%%%%%%%%%%%%%%%%%%%%%%%%%%%%%%%%%%%%%%%%%%%%%%%%%%%%%%%%%%%%%%%%%%%%%%%
%%%%%%%%%%%%%%%%%%%%%%%%%%%%%%%%%%%%%%%%%%%%%%%%%%%%%%%%%%%%%%%%%%%%%%%%%%%%%%%%%%%%%%%%%%%%%%%%%%%%%%%%%%%%

\section{The symmetric generating function for ASMs}
\label{sec: ASM gen fct}

An \emph{alternating sign matrix}, or \emph{ASM} for short, of size $n$ is an $n \times n$ matrix with entries $-1,0,1$ such that all row- and column-sums are equal to $1$ and in all rows and columns the non-zero entries alternate. See Figure \ref{fig: ASM and MT} (left) for an example of an ASM of size $6$. We denote by $\ASM_n$ the set of ASMs of size $n$. 
Following the convention of \cite[Eq. 18]{RobbinsRumsey86} and \cite{Fischer18}, we define the \emph{inversion number} $\inv$ and the \emph{complementary inversion number} $\inv^\prime$ of an ASM $A=(a_{i,j})_{1 \leq i,j \leq n}$ of size $n$ as
$$
\inv(A):= \sum_{1 \leq i^\prime < i \leq n \atop 1 \leq j^\prime \leq j \leq n}a_{i^\prime, j}a_{i,j^\prime} \quad \text{and} \quad
\inv^\prime(A):= \sum_{1 \leq i^\prime < i \leq n \atop 1 \leq j \leq j^\prime \leq n}a_{i^\prime, j}a_{i,j^\prime},
$$
and denote by $\mathcal{N}(A)$ the number of $-1$'s of $A$.  The number of $-1$ entries, the inversion number and the complementary inversion number of an ASM $A$ of size $n$ are connected by
\[
\mathcal{N}(A)+\inv(A)+\inv^\prime(A)=\binom{n}{2},
\]
which follows immediately by relating these statistics with the corresponding statistics on monotone triangles; this is described after Theorem \ref{thm: sym gen fct via Operatoren}.
It is easy to see that there is a unique $1$ entry in the top (resp. bottom) row of $A$. We denote by $\rho_T(A)$ the number of $0$ entries left of the unique $1$ in the top row, and by $\rho_B(A)$ the number of $0$ entries right of the unique $1$ in the bottom row.
For the example given in Figure \ref{fig: ASM and MT}, the five statistics are $(\mathcal{N}(A),\inv(A),\inv^\prime(A),\rho_T(A),\rho_B(A))=(2,6,7,3,2)$.
\begin{figure}[h]
\begin{center}
\[
\begin{pmatrix}
0 & 0 & 0 & 1 & 0 & 0 \\
0 & 1 & 0 & 0 & 0 & 0 \\
0 & 0 & 1 & -1 & 1 & 0 \\
1 & -1 & 0 & 0 & 0 & 1 \\
0 & 1 & 0 & 0 & 0 & 0 \\
0 & 0 & 0 & 1 & 0 & 0
\end{pmatrix}
\qquad \qquad
\begin{array}{c cc cc cc cc cc}
&&&&& \it{\cyan{4}} \\
&&&& 2 && \mathbf{\red{4}}\\
&&&\mathbf{\red{2}} && 3 && 5 \\
&& 1 && \it{\cyan{3}} && \it{\cyan{5}} && \it{\cyan{6}}\\
& 1 && 2 && 3 && \it{\cyan{5}} && \it{\cyan{6}}\\
1 && 2 && 3 && 4 && 5 && 6
\end{array}
\]
\end{center}
\caption{\label{fig: ASM and MT} An ASM of size $6$ and its corresponding monotone triangle, where the special entries are written in bold and red, and the right-leaning entries are in italic and blue.}
\end{figure}
\bigskip

A \emph{monotone triangle} with $n$ rows is a triangular array $(m_{i,j})_{1\leq j \leq i \leq n}$ of integers of the following form,
\[
\begin{array}{c c c c c c c c c c c}
&&&&& m_{1,1}\\
&&&& m_{2,1} && a_{2,2}\\
&&& \iddots && \cdots &&\ddots\\
&& \iddots && \iddots && \ddots &&\ddots\\
& m_{n-1,1} && m_{n-1,2} && \cdots  && \cdots  && m_{n-1,n-1}\\
m_{n,1} && m_{n,2} && m_{n,3} && \cdots  && \cdots && m_{n,n}
\end{array}
\]
such that the entries are weakly increasing along northeast and southeast diagonals, i.e., $m_{i+1,j}  \leq m_{i,j} \leq m_{i+1,j+1}$, and strictly increasing along rows. Given an ASM $A$ of size $n$, we obtain a monotone triangle by recording in the $i$-th row from top the indices of the columns with a positive partial column sum of the top $i$ rows of $A$. For an example see Figure \ref{fig: ASM and MT}. It is well-known that this map is a bijection between ASMs of size $n$ and monotone triangles with bottom row $1,2,\ldots,n$.
Each entry of a monotone triangle $M=(m_{i,j})_{1 \leq j \leq i \leq n}$ not in the bottom row is exactly of one of the following three types.
\begin{itemize}
\item An entry $m_{i,j}$ is called \emph{special} iff $m_{i+1,j} < m_{i,j} < m_{i+1,j+1}$.
\item An entry $m_{i,j}$ is called \emph{left-leaning} iff $m_{i,j}= m_{i+1,j}$.
\item An entry $m_{i,j}$ is called \emph{right-leaning} iff $m_{i,j}= m_{i+1,j+1}$.
\end{itemize}
For $1 \leq i \leq n-1$, we define the following statistics,
\begin{align*}
\begin{array}{lccl}
s_i(M) = \# \text{ of special entries in row }i,   &&&s(M) = \# \text{ of all special entries},\\
\,l_i(M) = \# \text{ of left-leaning entries in row }i,  &&&\,l(M) = \# \text{ of all left-leaning entries},\\
r_i(M) = \# \text{ of right-leaning entries in row }i,   &&&r(M) = \# \text{ of all right-leaning entries},
\end{array}
\end{align*} 
and set $s_0(M)=l_0(M)=r_0(M)=0$. In our running example in Figure \ref{fig: ASM and MT}, these statistics are
\begin{align*}
(s_i(M))_{1 \leq i \leq 5} &= (0,1,1,0,0), & s(M) = 2,\\
(l_i(M))_{1 \leq i \leq 5} &= (0,1,2,1,3), & l(M) = 7,\\
(r_i(M))_{1 \leq i \leq 5} &= (1,0,0,3,2), & r(M) = 6. 
\end{align*}
Finally, we set for $1 \leq i \leq n$
\begin{align*}
\widehat d_i(M) = \sum_{j=1}^i m_{i,j} - \sum_{j=1}^{i-1} m_{i-1,j} +r_{i-1}(M) -l_{i-1}(M)-1,
\end{align*}
and define the weight $\omega_M(u,v,w;\x)$ of a monotone triangle as
\[
\omega_M(u,v,w;\x) =  u^{r(M)} v^{l(M)}  \prod_i^{n} x_i^{\widehat d_i(M)} ( u x_i + w + v x_i^{-1} )^{s_{i-1}(M)} ,
\]
where $\x=(x_1,\ldots,x_n)$. In our running example in Figure \ref{fig: ASM and MT}, the weight $\omega_M(u,v,w;\x)$ is given by
\[
\omega_M(u,v,w;\x) =  u^6  v^7 x_1^3 x_2^2 x_3^2 x_4^2 x_5^3 x_6^2 ( u x_3 + w + v x_3^{-1})( u x_4 + w +v x_4^{-1}) 
\]
\begin{rem}
\label{rem: slightly different weight}
The weight $\omega_M(u,v,w;\x)$ is related to the weight $W_0(M)$, which is defined in \cite[p. 12]{AignerFischer2106.11568}, by the relation 
\[
\omega_M(u,v,w;\x) =W_0(M^\prime),
\]
where $M^\prime$ is the monotone triangle obtained by subtracting $1$ from all entries in $M$.
\end{rem}

For an ASM $A$, we set $\omega_A(u,v,w;\x)=\omega_M(u,v,w;\x)$, where $M$ is the corresponding monotone triangle. We call the generating function of ASMs with respect to the weight $\omega_A(u,v,w;\x)$ the \emph{symmetric generating function} for ASMs since it turns out to be a symmetric polynomial in $\x$.  As a special case of \cite[Theorem 3.1]{AignerFischer2106.11568}, we have the following theorem.

\begin{thm}
\label{thm: sym gen fct via Operatoren}
Let $E_x$ denote the \emph{shift operator} which is defined as $E_x f(x)=f(x+1)$.
The \emph{symmetric generating function} for ASMs of size $n$ is
\begin{equation}
\label{eq: sym gen fct via Operatoren}
\sum_{A \in \ASM_n} \omega_A(u,v,w;\x) = \left.\prod_{1 \leq i < j \leq n} \left(u E_{\lambda_i}  + w E_{\lambda_i}E_{\lambda_j}^{-1} +  v E_{\lambda_j}^{-1}  \right) s_{(\lambda_n,\ldots,\lambda_1)}(\x)\right|_{\lambda_i=i-1}.
\end{equation}
\end{thm}

Let $A$ be an ASM of size $n$ and $M$ its corresponding monotone triangle. By comparing the statistics for ASMs and monotone triangles, we have
\begin{align*}
\inv(A) &= r(M), \\
\inv^\prime(A) &= l(M),\\
\mathcal{N}(A) &= s(M),\\
\rho_T(A) &= \widehat{d}_1(M),\\
\rho_B(A) &= \widehat{d}_n(M),
\end{align*}
where the last two identities follow directly from the definitions and for the first three compare to \cite{AignerFischer2106.11568}.
Since the bottom row of $M$ is $1,2,\ldots,n$, there are no special entries in row $n-1$. Further there are no special entries in row $0$, since this row has no entries. Therefore, the symmetric generating function specialises to
\begin{equation}
\label{eq: specialisation of sym gen fct}
\left. \sum_{A \in \ASM_n} \omega_A(u,v,w;\x) \right|_{x_2 = \ldots =x_{n-1}=1} =
\sum_{A \in \ASM_n}  (u+ v+w)^{\mathcal{N}(A)}  u^{\inv(A)}  v^{\inv^\prime(A)} x_1^{\rho_T(A)} x_n^{\rho_B(A)}.
\end{equation}
\bigskip

Theorem~\ref{eq: sym gen fct via Operatoren} arose naturally from a constant term formulation of the operator formula in \cite{Fischer06} for monotone triangles with bottom row $1,2,\ldots,n$ (it is generalised to arbitrary bottom rows in \cite[Theorem 3.1]{AignerFischer2106.11568}). The purpose of the following digression is to relate it to a function that appeared in connection with the six-vertex model. This interesting relation was brought to our attention by a referee of the FPSAC submission, and we wish to thank her/him for sharing this insight.

A \emph{configuration of the six-vertex model} of size $n$ is an orientation of the $n \times n$ grid with $n$ external edges\footnote{An external edge is an edge with only one incident vertex.} on each side such that for each vertex the indegree equals the outdegree. We restrict ourselves to configurations, where the external edges on the top and bottom are oriented outwards, and on the left and right are oriented inwards; this is called the \emph{domain wall boundary condition} (DWBC). 
\begin{figure}[h]
\begin{center}
\begin{tikzpicture}
\begin{scope}[scale=0.85]
\VertexRU{0}{0}
\VertexLD{2}{0}
\VertexLU{4.5}{0}
\VertexRD{6.5}{0}
\VertexOne{9}{0}
\VertexMinus{11}{0}
\node at (0,-1.5) {$(1)$};
\node at (2,-1.5) {$(2)$};
\node at (4.5,-1.5) {$(3)$};
\node at (6.5,-1.5) {$(4)$};
\node at (9,-1.5) {$(5)$};
\node at (11,-1.5) {$(6)$};
\end{scope}
\end{tikzpicture}
\caption{\label{fig: 6-vert} The six possible configurations at a vertex.} 
\end{center}
\end{figure}
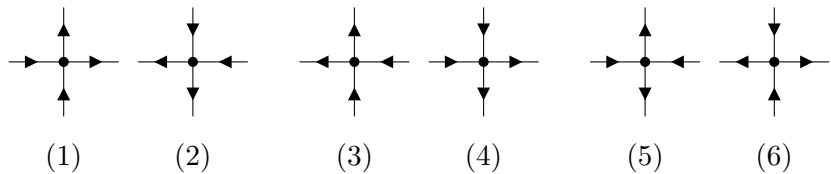
It is well known that configurations of the six-vertex model with DWBC are mapped bijectively to ASMs by replacing the fifth vertex configurations in Figure \ref{fig: 6-vert} by a $1$ entry, the sixth configuration by a $-1$ entry and the other configurations by $0$ entries. For an example see Figure \ref{fig: asm to 6vert}.
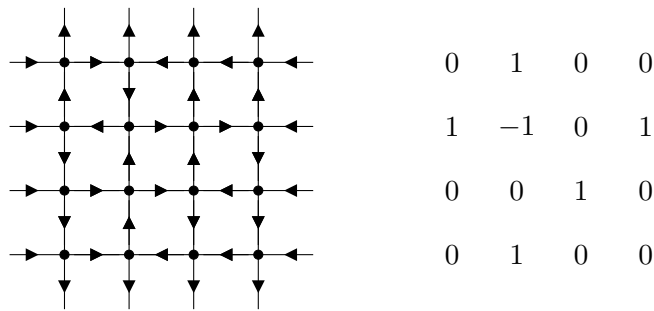
\begin{figure}
\begin{center}
\begin{tikzpicture}
\begin{scope}[scale=0.85]
\VertexRU{0}{0}
\VertexRU{2}{-1}
\VertexRU{1}{-2}
\VertexLU{2}{0}
\VertexLU{3}{0}
\VertexRD{0}{-2}
\VertexRD{0}{-3}
\VertexLD{2}{-3}
\VertexLD{3}{-3}
\VertexLD{3}{-2}
\VertexOne{1}{0}
\VertexOne{0}{-1}
\VertexOne{3}{-1}
\VertexOne{2}{-2}
\VertexOne{1}{-3}
\VertexMinus{1}{-1}
\begin{scope}[xshift=6cm]
\node at (0,0) {$0$};
\node at (1,0) {$1$};
\node at (2,0) {$0$};
\node at (3,0) {$0$};
\node at (0,-1) {$1$};
\node at (1,-1) {$-1$};
\node at (2,-1) {$0$};
\node at (3,-1) {$1$};
\node at (0,-2) {$0$};
\node at (1,-2) {$0$};
\node at (2,-2) {$1$};
\node at (3,-2) {$0$};
\node at (0,-3) {$0$};
\node at (1,-3) {$1$};
\node at (2,-3) {$0$};
\node at (3,-3) {$0$};
\end{scope}
\end{scope}
\end{tikzpicture}
\end{center}
\caption{\label{fig: asm to 6vert} A configuration of the six-vertex model with DWBC and its corresponding ASM.}
\end{figure}
For an ASM $A$, we denote by $\nu_i(A)$ the number of configurations of type $(1)$ and $(2)$ in the $i$-th row of the corresponding six-vertex configuration and by $\mu_i(A)$ the number of configurations of type $(6)$ in row $i$. 
In \cite{Behrend13}, Behrend considered the following generating function of ASMs
\begin{multline*}
X_n(u,w;x_1,\ldots,x_n) \\= Z_n^{1,2,\ldots,n}(u,w;x_1,\ldots,x_n;u x_1^2 + (w-u-1)x_1+1, \ldots, 	u x_n^2 + (w-u-1)x_n+1) \\
= \sum_{A \in \ASM_n} u^{\inv(A)} \prod_{i=1}^n x_i^{\nu_i(A)} (u x_i^2 + (w-u-1)x_i+1)^{\mu_i(A)}. 
\end{multline*}
For the definitions of $X_n$ and $Z_n$, see \cite[Eqs. 67, 70, 73]{Behrend13} and, for the statistics compare to \cite[Eq. 2, 113]{Behrend13}. In the following, we show that the function $X_n$ satisfies
\begin{equation}
\label{eq: connection to X}
\sum_{A \in \ASM_n}  \omega_A(u,1,w;\x)  = X_n(u,1+u+w;\x).
\end{equation}
For an ASM $A$, let $M=(m_{i,j})$ be the monotone triangle associated to $A$. The equation \eqref{eq: connection to X} is an easy consequence of the identities $s_{i-1}(M)=\mu_i(A)$ and $\nu_i(A) = \widehat{d}_i(M)-s_{i-1}(M)$. The first identity follows directly from the bijections between monotone triangles, ASMs and configurations of the six-vertex model. In the remainder of this section, we prove the second identity.

Let $a_0,\ldots,a_l$ (resp. $b_1,\ldots,b_l$) be the positions of the $1$ (resp. $-1$) entries in row $i$. By the definition of the bijection between monotone triangles and ASMs, we have
\begin{align*}
\{a_0,\ldots,a_l\} &= \{m_{i,1}, \ldots ,m_{i,i} \} \setminus \{ m_{i-1,1}, \ldots , m_{i-1,i-1}\},\\
\{b_1,\ldots,b_l\} &= \{ m_{i-1,1}, \ldots , m_{i-1,i-1}\} \setminus \{m_{i,1}, \ldots ,m_{i,i} \}.
\end{align*}
Note that the second equality implies $l=s_{i-1}(M)$.
%In the corresponding six-vertex configuration, the vertex configurations of type $(1)$ correspond to the $0$ entries either left of the first $1$ or between a $-1$ and the following $1$ entry with the property that the entries in the same column and above the $0$ sum to $0$.
%It is easy to see that these columns correspond to the left-leaning entries of $M$ in row $i-1$. Configurations of type $(2)$ correspond to $0$ entries between a $1$ and the following $-1$ entry with the property that the entries in the same column and above the $0$ sum to $1$.  These positions correspond to the right-leaning entries in $M$ in row $i-1$. Putting this all together, we have
In the corresponding six-vertex configuration, the vertex configurations of type $(1)$ correspond to $0$ entries in the ASM that satisfy the following two conditions: (a) they are left of the first $1$ or between a $-1$ and the following $1$, and (b) the entries in the same column and above the $0$ sum to $0$. There are $(a_0-1) + \sum_{j=1}^l (a_j-b_j-1)$ entries satisfying condition (a). On the other hand, it is not difficult to see that the $0$ entries which satisfy (a) but not (b) are exactly in the columns corresponding to a left-leaning entry of $M$ in row $i-1$, i.e., there are $l_{i-1}(M)$ such entries. Configurations of type $(2)$ correspond to $0$ entries between a $1$ and the following $-1$ entry with the property that the entries in the same column and above the $0$ sum to $1$. These positions correspond to the right-leaning entries in $M$ in row $i-1$, hence there are $r_{i-1}(M)$ such entries. Putting this all together, we have
\begin{multline*}
\nu_i(A) = (a_0-1) + \sum_{j=1}^l (a_j-b_j-1)  - l_{i-1}(M) + r_{i-1}(M)\\
= \sum_{j=1}^i m_{i,j} - \sum_{j=1}^{i-1} m_{i-1,j}  -s_{i-1}(M)-1- l_{i-1}(M) + r_{i-1}(M)
=\widehat{d}_i(M)  -s_{i-1}(M).
\end{multline*}

%%%%%%%%%%%%%%%%%%%%%%%%%%%%%%%%%%%%%%%%%%%%%%%%%%%%%%%%%%%%%%%%%%%%%%%%%%%%%%%%%%%%%%%%%%%%%%%%%%%%%%%%%%%%
%%%%%%%%%%%%%%%%%%%%%%%%%%%%%%%%%%%%%%%%%%%%%%%%%%%%%%%%%%%%%%%%%%%%%%%%%%%%%%%%%%%%%%%%%%%%%%%%%%%%%%%%%%%%
%                                            4) Antisym formula
%%%%%%%%%%%%%%%%%%%%%%%%%%%%%%%%%%%%%%%%%%%%%%%%%%%%%%%%%%%%%%%%%%%%%%%%%%%%%%%%%%%%%%%%%%%%%%%%%%%%%%%%%%%%
%%%%%%%%%%%%%%%%%%%%%%%%%%%%%%%%%%%%%%%%%%%%%%%%%%%%%%%%%%%%%%%%%%%%%%%%%%%%%%%%%%%%%%%%%%%%%%%%%%%%%%%%%%%%

\section{An antisymmetriser to determinant formula}
\label{sec: antisym to det}
In this section we provide a fundamental tool for the proof of Theorem~\ref{thm: main thm 1}. We present both a non-combinatorial proof and a combinatorial proof for it. More applications of it appeared in \cite{AignerFischer2106.11568}.

\begin{lem}  
\label{lem: general}
Let $n \ge 1$, and $\mathbb{X}=(X_1,\ldots,X_n), \mathbb{Y}=(Y_1,\ldots,Y_n)$ be indeterminants.
 Then
$$
\widehat{\asym}  \left[  \prod_{1 \le i \le j \le n} (Y_j-X_i)     \right]
= \det_{1 \le i, j \le n} \left( Y_i^j - X_i^j \right),
$$ 
with 
$$
\widehat{\asym} \left[f(\mathbb{X};\mathbb{Y})\right] = \sum_{\sigma \in {\mathcal{S}_n}} \sgn \sigma f(X_{\sigma(1)},\ldots,X_{\sigma(n)};Y_{\sigma(1)},\ldots,Y_{\sigma(n)}).
$$ 
\end{lem}

\begin{proof}[First proof]  
Since we aim that proving the equality of two polynomials in $X_1,\ldots,X_n,Y_1,\ldots,Y_n$, standard arguments imply that it suffices to consider the case when $X_1,\ldots,X_n,Y_1,\ldots,Y_n$ are algebraically independent. In particularly, we may assume $\det_{1 \le i, j \le n} \left( Y_i^j - X_i^j \right) \not=0$, which will be useful below.

%The alternative is not  longer and gives one whole argument, but may sound stupid. Therefore, I have a preference for the first option now.

%We can assume $\det_{1 \le i, j \le n} \left( Y_i^j - X_i^j \right) \not=0$ because otherwise we can find $\epsilon_{i,m} \in \mathbb{R}$ such that $\lim_{m \to \infty} \epsilon_{i,m} = 0$ and  $\det_{1 \le i, j \le n} \left( (Y_i+\epsilon_{i,m})^j - X_i^j \right) \not=0$, and the result then follows from the case $\det_{1 \le i, j \le n} \left( Y_i^j - X_i^j \right) \not=0$ by considering the limit $m \to \infty$.

The proof is by induction with respect to $n$. The result is obvious for $n=1$. 
Let $L_n(\mathbb{X};\mathbb{Y})$, $R_n(\mathbb{X};\mathbb{Y})$ denote the left- and right-hand side of the identity in the statement, respectively. By the induction hypothesis, we can assume 
$ L_{n-1}(Y_1,\ldots,Y_{n-1};X_1,\ldots,X_{n-1}) \allowbreak \break= R_{n-1}(Y_1,\ldots,Y_{n-1};X_1,\ldots,X_{n-1})$.
We show that both $L_n(\mathbb{X};\mathbb{Y})$ and $R_n(\mathbb{X};\mathbb{Y})$ can be computed recursively using 
$L_{n-1}(X_1,\ldots,X_{n-1};Y_1,\ldots,Y_{n-1})$ and $R_{n-1}(X_1,\ldots,X_{n-1};Y_1,\ldots,Y_{n-1})$, respectively, with the same recursion.
For the left-hand side, we have 
$$
L_n(\mathbb{X};\mathbb{Y})
=  \sum_{i=1}^{n} (-1)^{i+1}  \left( \prod_{k=1}^n (Y_k-X_i)  \right) 
L_{n-1}(X_1,\ldots,\widehat{X_i},\ldots,X_n;Y_1,\ldots,\widehat{Y_i},\ldots,Y_n),
$$
where $\widehat{X_i}$ and $\widehat{Y_i}$ means that $X_i$ and $Y_i$ are omitted.
In order to deal with the right-hand side, we first observe
\begin{equation}
\label{fundamentalidentity}
\sum_{j=0}^{n} (Y_i^j - X_i^j) e_{n-j}(-Y_1,\ldots,-Y_n) = (-1)^{n-1} \prod_{k=1}^{n} (Y_k-X_i),
\end{equation}
where $e_{j}(Y_1,\ldots,Y_n)$ denotes the $j$-th elementary symmetric function. Note that the summand for $j=0$ on the left-hand side is actually $0$, and can therefore be omitted.
Now consider the following system of linear equations with $n$ unknowns $c_j(\mathbb{X};\mathbb{Y})$, $1 \le j \le n$, and 
$n$ equations.
$$
\sum_{j=1}^{n} (Y_i^j - X_i^j)  c_j(\mathbb{X};\mathbb{Y})
= (-1)^{n-1} \prod_{k=1}^{n} (Y_k-X_i), \quad 1 \le i \le n.
$$
The determinant of this system of equations is obviously $R_n(\mathbb{X};\mathbb{Y})$ and can be assumed to be non-zero. By \eqref{fundamentalidentity}, we know that the unique solution of this system is given by 
$
c_j(\mathbb{X};\mathbb{Y}) = e_{n-j}(-Y_1,\ldots,-Y_n).
$
On the other hand, by Cramer's rule, 
$$
c_n(\mathbb{X};\mathbb{Y}) = \frac{\det \limits_{1 \le i, j \le n} \left( \begin{cases}  Y_i^j - X_i^j, & \text{if $j<n$} \\
  (-1)^{n-1} \prod\limits_{k=1}^{n} (Y_k-X_i) , & \text{if $j=n$}  \end{cases} \right)}{R_n(\mathbb{X};\mathbb{Y})}.
$$
The assertion now follows from $c_n(\mathbb{X};\mathbb{Y}) = e_{0}(-Y_1,\ldots,-Y_n)=1$ and expanding the determinant in the numerator with respect to the last column.
\end{proof}

We also provide a combinatorial proof. For this purpose, we need a number of definitions to reformulate the problem so that it is accessible from a combinatorial point of view.

Replacing $X_i \to -X_i$, we need to show 
\begin{equation}
\label{reform} 
\widehat{\asym}  \left[  \prod_{1 \le i \le j \le n} (X_i+Y_j)     \right] =
\det_{1 \le i, j \le n} \left( Y_i^j + (-1)^{j+1} X_i^j \right).
\end{equation} 

Let $L_n$ denote the graph that is obtained from the complete simple graph on the vertex set $\{1,2,\ldots,n\}$ by adding one loop at each vertex. We consider orientations of $L_n$ and imagine the vertices $1,2,\ldots,n$ to be arranged on a horizontal line. We say an edge is oriented from left to right if it is oriented from the smaller vertex $i$ to the larger vertex $j$ (and write $i \rightarrow j$) and from right to left otherwise ($i \leftarrow j$).
It will be convenient to have two possible orientations for loops also, say, from left to right (indicated as $i \rightarrow i$) and from right to left (indicated as $i \leftarrow i$), so that there are in total $2^{\binom{n+1}{2}}$ orientations of $L_n$.  The set of all orientations of $L_n$ is denoted by $\mathcal{O}_n$. An example is provided in Figure~\ref{orient}. 

Now each monomial in the expansion of $\prod_{1 \le i \le j \le n} (X_i+Y_j)$ clearly corresponds to an orientation of $L_n$ as follows: For $i \le j$, we let $i \rightarrow j$ if we pick $X_i$ in $X_i+Y_j$ and $i \leftarrow j$ if we pick $Y_j$. Thus, the weight of an orientation $O \in \mathcal{O}_n$ is defined as
$$
w(O)=\prod_{i=1}^{n} X_i^{\# \{j \ge i:i \rightarrow j\}} Y_i^{\# \{j \le i:j \leftarrow i\}},
$$
so that
$\sum_{O \in \mathcal{O}_n} w(O) = \prod_{1 \le i \le j \le n} (X_i+Y_j)$. The weight in our example is 
$X_1^7 X_2^5 X_3^4 X_4 Y_5^2 Y_6^3 Y_7^6$. 

\begin{figure}
\scalebox{0.4}{\includegraphics{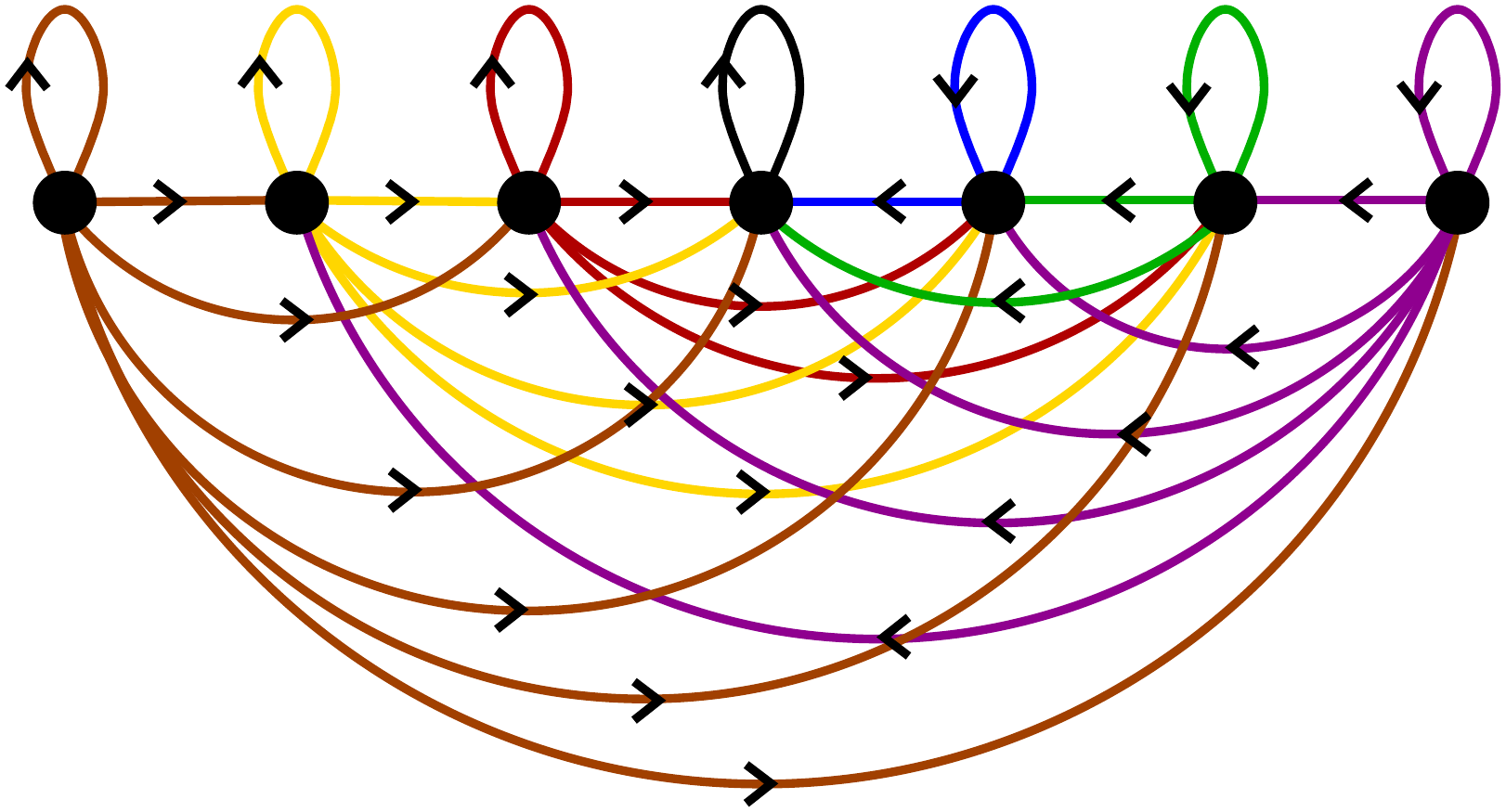}}
\caption{\label{orient} An orientation of $L_7$ that is in $\mathcal{P}_7$.} 
\end{figure} 

We consider a subset $\mathcal{P}_n$ of orientations in $\mathcal{O}_n$ that will provide a combinatorial interpretation for the 
right-hand side of \eqref{reform}. The definition is recursive: We have $\mathcal{P}_1 = \mathcal{O}_1$, and, for $n>1$, $\mathcal{P}_n$ is partitioned into two set: 
\begin{itemize} 
\item either all edges incident with $n$ are oriented away from $n$ (necessarily to the left) and the restriction of the orientation to $\{1,2,\ldots,n-1\}$ is in $\mathcal{P}_{n-1}$, 
\item 
or all edges incident with $1$ are oriented away from $1$ 
(necessarily to the right) and the restriction of the orientation to $\{2,3,\ldots,n\}$ is in $\mathcal{P}_{n-1}$ with vertices renamed through a shift by $1$. 
\end{itemize} 
There are clearly $2^{n}$ such orientations in $\mathcal{P}_{n}$ and the orientation 
in Figure~\ref{orient} is in $\mathcal{P}_7$.

Orientations in $\mathcal{P}_{n}$ can be encoded by a linear order of the vertices $1,2,\ldots,n$ that is induced by the inductive build-up of the orientations together with the orientation of the loop of the first vertex in the list. In the example in Figure~\ref{orient}, the order is $4 \, 5 \, 6 \, 3 \, 2 \, 7 \, 1$. This encoding has the following features.
\begin{itemize} 
\item Each vertex in the list is either greater than all its predecessors in the list or smaller than all its predecessors. 
\item The orientation is obtained from the list as follows: The edges are oriented away from each vertex to its predecessors in the linear order, and the loop of a vertex different from the first vertex in the list is oriented 
from left to right if this vertex is smaller than all its predecessors and from right to left otherwise. The orientation of the loop of the first vertex is given.
\item The weight can easily be computed as follows: The exponent of $X_i$ or $Y_i$ is the position of vertex $i$ in the list.
\end{itemize}

It follows that for each orientation in $\mathcal{P}_n$, the set $\{1,2,\ldots,n\}$ can be partitioned into maximal intervals of integers that are either added consecutively ``from above'' (upper sections) or added consecutively ``from below'' (lower sections) in the recursive procedure.  
More formally, there is a strictly increasing sequence of integers $i_0 < i_1 < i_2 < \ldots < n+1=i_s$ and a strictly decreasing sequence of integers $i_0-1=j_0 > j_1 > j_2 > \ldots > 0=j_t$ such that the linear order is 
$$
\color{red} i_0, i_0+1,i_0+2,\ldots,i_1-1, \color{blue} j_0,j_0-1,\ldots,j_1+1, \color{red} i_1,i_1+1,\ldots,i_2-1, \color{blue} j_1,j_1-1,\ldots,j_2+1,\color{red} \ldots \color{black} .
$$
The vertices greater than $i_0$ have their outgoing edges all to the left, while the vertices smaller than $i_0$ have their outgoing edges all to the right. The intervals $[i_0,i_1-1],[i_1,i_2-1],[i_2,i_3-1],\ldots$ are said to be the upper sections, while the intervals $[j_1+1,j_0],[j_2+1,j_1],[j_3+1,j_2],\ldots$ are said to be the lower sections. The only exceptional case happens if $i_0 \rightarrow i_0$: If $i_1=i_0+1$, then $[j_1+1,i_0]$ is a lower section and if $i_1>i_0+1$, then $[i_0,i_0]$ is a lower section and $[i_0+1,i_1-1]$ is an upper section. In our example, $(i_0,i_1,i_2)=(4,7,8)$ and $(j_0,j_1,j_3)=(3,1,0)$. Here we are in the exceptional case, so that $[1,1],[2,3],[4,4]$ are the lower sections and $[5,6], [7,7]$ are the upper sections.  

The claim \eqref{reform}  is equivalent to 
\begin{equation}
\label{reform1}  
\widehat{\asym}  \left[ \sum_{O \in \mathcal{O}_n} w(O) \right] = \widehat{\asym}  \left[ \sum_{O \in \mathcal{P}_n} w(O) \right]. 
\end{equation} 
In order to see this equivalence, we need to show 
$$
\widehat{\asym}  \left[ \sum_{O \in \mathcal{P}_n} w(O) \right] = \det_{1 \le i, j \le n} \left( Y_i^j + (-1)^{j+1} X_i^j \right), 
$$
which we do by induction with respect to $n$. The case $n=1$ is easy. By definition, 
$$
 \sum_{O \in \mathcal{P}_n} w(O) = Y_n^{n}  \sum_{O \in \mathcal{P}_{n-1}} w(O) + 
X_1^{n} (1,2,\ldots,n) \left[ \sum_{O \in \mathcal{P}_{n-1}} w(O) \right], 
$$
where $(1,2,\ldots,n)$ denotes the cyclic permutation that sends $i \to i+1 \hspace{-2mm} \mod n$ and acts on $X_i$ and $Y_i$ simultaneously. Therefore, 
\begin{multline*} 
\widehat{\asym} \left[ \sum_{O \in \mathcal{P}_n} w(O) \right] 
\\ =
\sum_{\sigma \in \mathcal{S}_n} \sgn \sigma \cdot \sigma \left[ Y_n^{n}  \sum_{O \in \mathcal{P}_{n-1}} w(O) \right] +
 \sum_{\sigma \in \mathcal{S}_n} \sgn \sigma \cdot \sigma \left[ X_1^{n} (1,2,\ldots,n) \left[ \sum_{O \in \mathcal{P}_{n-1}} w(O) \right] \right]. 
\end{multline*} 
By the induction hypothesis, this is equal to 
\begin{multline*} 
\sum_{k=1}^n (-1)^{n+k} Y_k^n \det_{i \in \{1,2,\ldots,n\} \setminus \{k\} \atop 1 \le j \le n-1} (Y_i^j + (-1)^{j+1} X_i^j) \\+ 
\sum_{k=1}^n (-1)^{1+k} X_k^n \det_{i \in \{1,2,\ldots,n\} \setminus \{k\} \atop 1 \le j \le n-1} (Y_i^j + (-1)^{j+1} X_i^j) \\
= \sum_{k=1}^n (-1)^{n+k} (Y_k^n + (-1)^{n+1} X_k^n)  \det_{i \in \{1,2,\ldots,n\} \setminus \{k\} \atop 1 \le j \le n-1} (Y_i^j + (-1)^{j+1} X_i^j)=\det_{1 \le i, j \le n} \left( Y_i^j + (-1)^{j+1} X_i^j \right), 
\end{multline*} 
where the last equality follows from expanding with respect to the last column.

Rephrasing \eqref{reform1}, we need to show 
\begin{equation}
\label{zero}
\widehat{\asym}  \left[ \sum_{O \in \mathcal{R}_n} w(O) \right] = 0, 
\end{equation} 
with $\mathcal{R}_n = \mathcal{O}_n \setminus \mathcal{P}_n$, and we provide a combinatorial proof for this identity.

\begin{proof}[Combinatorial proof of \eqref{zero}]
It suffices to find an involution on $\mathcal{R}_n$ such that when orientation $O_1$ is paired with $O_2$ under this involution, then there exists a transposition $\tau \in \mathcal{S}_n$ with 
$w(O_2) = \tau \, w(O_1)$. 

We will use of the following notation: For an orientation $O \in \mathcal{O}_n$ and a subset 
$S \subseteq [n]$, we let $O|_{S}$ denote the restriction of $O$ to the subgraph of $L_n$ induced by $S$. We may also identify this with an element of $\mathcal{O}_{|S|}$ in a natural way, i.e., by using the isomorphism between 
$L_{|S|}$ and the restriction of $L_n$ to $S$ that is induced the unique order-preserving bijection between $[|S|]$ and $S$. 

Now suppose that $O \in \mathcal{R}_n$ and let $m$ be minimal such that $O|_{[m]} \in \mathcal{R}_{m}$. It follows that 
$O|_{[m-1]} \in \mathcal{P}_{m-1}$. When referring to lower sections in the following, we mean lower sections of the restriction of $O|_{[m-1]}$.  First we get rid of the following case.

\medskip

{\bf Step 1. There is a lower section $[p,q]$ and an integer $k$ with $p \le k  < q$ such that $k \leftarrow m$ and $k+1 \rightarrow m$.} 

The weight of $O|_{[m]}$ is invariant under applying the transposition $(k,k+1)$: for $r \in \{1,2,\ldots,m-1\} \setminus \{k,k+1\}$, the edges $\{k,r\},\{k+1,r\}$ have the same orientation, since they are in the same lower section. The weight that comes from the restriction to 
$\{k,k+1,m\}$ is either $X_k^2 X_{k+1}^2 X_m Y_m$ (if $m \rightarrow m$) or  $X_k^2 X_{k+1}^2 Y_m^2$ (if $m \leftarrow m$). We ``exchange the neighbourhoods'' of $k,k+1$ in $\{m+1,m+2,\ldots,n\}$: For all $j \in \{m+1,m+2,\ldots,n\}$, we have $k \rightarrow j$ in the new orientation iff $k+1 \rightarrow j$ in the old orientation, and we have $k \leftarrow j$ in the new orientation iff $k+1 \leftarrow j$ in the old orientation. The transposition $\tau$ is equal to $(k,k+1)$. 

Clearly, the so-obtained orientation is again of the same type (i.e., there is a lower section with such an integer $k$), and the map is an involution.

\medskip

Therefore, we can assume from now on that for each lower section $[p,q]$, there is a $k$ with $p-1 \le k \le q$ such that $p,p+1,\ldots,k \rightarrow m$ and $k+1,k+2,\ldots,q \leftarrow m$. We say that a lower section is normal if this is satisfied.

The idea of the remainder of the proof is roughly as follows: In the restriction $O|_{[m-1]}$, we consider for each vertex the number of left-pointing edges. From right to left, this is a strictly decreasing sequence of numbers, until these numbers are eventually $0$ for the remaining vertices. We compare them to the number of left-pointing edges from $m$. The typical case is that this number is between the numbers for two adjacent vertices $i,i+1$ in $\{1,2,\ldots,m-1\}$. It is then possible to let $\tau=(i,m)$ or $\tau=(i+1,m)$. Which of the two cases has to be chosen depends on the lower section between $i$ and $i+1$ in the total order of $1,2,\ldots,m-1$, more precisely on the $k$ just described that ``makes'' it into a normal section. The non-typical exceptional cases (such as for instance when $m$ has no left-pointing edges) makes the proof  involved.

\medskip

In the following, we let $\ell_i$ denote the number of left-pointing edges away from $i$. Next we rule out the following case.

{\bf  Step 2. There is an $i \in \{1,2,\ldots,m-1\}$ with $0 \not= \ell_i=\ell_m$.} 

We need to consider two cases here. 

\emph{Case 1: $i \leftarrow m \leftarrow m$ or $i \rightarrow m \rightarrow m$.} Note that within $\{1,2,\ldots,m\}$ the contribution of the vertices $i$ and $m$ to the weight is $Y_i^{\ell_m} Y_m^{\ell_m}$ in the first case and $X_i X_m Y_i^{\ell_m} Y_m^{\ell_m}$ in the second case. We only need to exchange the neighbourhood of $i$ and $m$ for vertices in $\{m+1,m+2,\ldots,n\}$. 

\emph{Case 2: $i \leftarrow m \rightarrow m$ or $i \rightarrow m \leftarrow m$.} Note that within $\{1,2,\ldots,m\}$ the contribution of $i$ and $m$ to the weight is $X_m Y_i^{\ell_m} Y_m^{\ell_m}$ in the first case and $X_i Y_i^{\ell_m} Y_m^{\ell_m}$ in the second case. We transform the cases into one another, and exchange the neighbourhoods of $i$ and $m$ in the vertex set $\{m+1,m+2,\ldots,n\}$. 

The transposition $\tau$ is equal to $(i,m)$. Note that orientations of edges incident with vertices in lower sections are not changed, and, therefore, all lower sections are still normal. Also note that we clearly stay within the type of orientations under consideration since the number of left-pointing edges from $i$ and $m$ does not change. The map is clearly an involution. 

\medskip

The only case that remains is the following.

\smallskip

{\bf  Step 3. We have $\ell_m \not= \ell_i$ for all $i \in [m-1]$ or $\ell_m=0$.} 

Since $O|_{[m-1]} \in \mathcal{P}_{m-1}$, we have $\ell_{m-1} > \ell_{m-2} > \ldots > \ell_{t} >0$, where $t$ is the smallest integer in an upper section (setting $t=\infty$ if $t$ does not exist). The case $\ell_m=0$ as well as some instances of the cases that 
$\ell_m > \ell_{m-1}$ and $\ell_{t} > \ell_m$  are dealt with after Cases A and Cases B.

\medskip

For now we assume that there exist 
$i,i+1 \in \{t,t+1,\ldots,m-1\}$ such that $\ell_{i+1} > \ell_{m} > \ell_{i}$. The transposition $\tau$ will be either $(i,m)$ or $(i+1,m)$. Let $[p,q]$ be the lower section that appears in the linear order of $[m-1]$ induced by $O|_{[m-1]}$ 
between $i$ and $i+1$ (which are by assumption contained in different upper sections, since $\ell_{i+1}-\ell_{i}>1$) so that this part of the linear order reads as
$$
i,q,q-1,\ldots,p,i+1,
$$ 
and let $k$ be such that $p,p+1,\ldots,k \rightarrow m$ and $k+1,k+2,\ldots,q \leftarrow m$ (such a $k$ exists because all lower sections are normal).  Since $(\ell_{i+1}-\ell_m) + (\ell_m-\ell_i) = \ell_{i+1}-\ell_i = q-p+2=(q-k)+(k-p+1)+1$, we have either $\ell_m-\ell_i \le q-k$ or $\ell_{i+1}-\ell_m \le k-p+1$ but not both. 

\medskip

\emph{Case A: $\ell_m-\ell_{i} \le q-k$} 

In this case, we change the linear order for $O|_{[m-1]}$ so that 
$$
i,q,q-1,\ldots,p,i+1 \Rightarrow q,q-1,\ldots,q-(\ell_m-\ell_{i})+1, i,q-(\ell_m-\ell_i),\ldots,p,i+1
$$ 
to the effect that $X_q X_{q-1} \ldots X_{q-(\ell_m-\ell_i)+1}$ in the weight is replaced by $Y_i^{\ell_m-\ell_i}$
and change 
$$
q-(\ell_m-\ell_{i})+1,q-(\ell_m-\ell_{i})+2,\ldots,q \leftarrow m \Rightarrow q-(\ell_m-\ell_{i})+1,q-(\ell_m-\ell_{i})+2,\ldots,q \rightarrow m, 
$$
to the effect that $Y_m^{\ell_m-\ell_i}$ in $Y_m^{\ell_m} = Y_m^{\ell_i} Y_m^{\ell_m-\ell_i}$ is replaced by 
$X_q X_{q-1} \ldots X_{q-(\ell_m-\ell_i)+1}$. 

In addition, in analogy to Case~2, we transform the case  
$i \leftarrow m \rightarrow m$  into $i \rightarrow m \leftarrow m$, and vice versa. There is no such transformation if $i \leftarrow m \leftarrow m$ or $i \rightarrow m \rightarrow m$ (as in Case~1). Finally, we exchange the neighbourhood of $i$ and $m$ in $\{m+1,m+2,\ldots,n\}$.

Note that still all lower sections are normal and the transposition $\tau$ is equal to 
$(i,m)$. 

We apply this case also if $i=m-1$ (but still $\ell_m-\ell_{m-1} \le q-k$). As $\ell_m > \ell_i > 0$, we automatically exclude 
$\ell_m=0$ here.

\medskip

\emph{Case B:  $\ell_{i+1}-\ell_m \le k-p+1$} 

In this case, we change the linear order for $O|_{[m-1]}$ so that 
$$
i,q,q-1,\ldots,p,i+1 \Rightarrow i,q,q-1,\ldots,p+\ell_{i+1}-\ell_m,i+1,p+\ell_{i+1}-\ell_m-1\ldots,p+1,p
$$
to the effect that $Y_{i+1}^{\ell_{i+1}-\ell_m}$ in $Y_{i+1}^{\ell_{i+1}} = Y_{i+1}^{\ell_{i+1}-\ell_m} Y_{i+1}^{\ell_m}$ is replaced by $X_p X_{p+1} \ldots X_{p+l_{i+1}-\ell_m-1}$ 
and change 
$$
p,p+1,\ldots,p+\ell_{i+1}-\ell_m-1 \rightarrow m \Rightarrow p,p+1,\ldots,p+\ell_{i+1}-\ell_m-1 \leftarrow m
$$
to the effect that $X_p X_{p+1} \ldots X_{p+\ell_{i+1}-\ell_m-1}$ is replaced by $Y_m^{\ell_{i+1}-\ell_m}$. In addition, we have 
again $i+1 \leftarrow m \rightarrow m \Leftrightarrow  i+1 \rightarrow m \leftarrow m$, and exchange the neighbourhood of $i+1$ and $m$ in 
$\{m+1,m+2,\ldots,n\}$.

Again all lower sections are still normal and the transposition $\tau$ is
$(i+1,m)$. 

We apply this case also if $i+1=t$ (but still $\ell_{t}-\ell_m \le k-p+1$). As $\ell_{t} \ge q-p+2 > k-p+1$, we automatically exclude $\ell_m=0$ also here.

We leave it to the reader to check that  Cases A and B ``match each other'': if we start with an orientation that falls under Case~A, it is transformed into one that falls under Case~B, and is then transformed into the original orientation, and vice versa. Therefore, we only need to figure out which cases are left and find an involution with the required property on them.

\medskip

{\bf Step 4.} We claim that the following two types are left.

\begin{enumerate}
\item $\ell_m=0$
\item Suppose $q$ is the first element in the list of the encoding of $O|_{[m-1]}$, then $q \rightarrow q$ 
and for the rightmost lower section $[p,q]$, there exists a $k$ with $p-1 \le k < q$ such that $i \leftarrow m$ iff $i \in [k+1,q]$. 
\end{enumerate} 

We will see that these cases are turned into one another under our involution. There is also no intersection as $\ell>0$ in the second case, since $[k+1,q]$ is not empty.

\medskip

{\bf (1) and (2) have not been dealt before:} This is obvious for (1). As for (2), we have that $\ell_m < \ell_t$ or $t=\infty$: if $t\not=\infty$, then $t=q+1$, $\ell_{q+1}=q-p+2$ and $\ell_m=q-k < q-p+2$, so it suffices to check $\ell_{q+1} - \ell_m  > k-p+1$ (because otherwise the case would have been dealt with in Case B), which is obviously satisfied. On the other hand, if $t=\infty$, then this case has also not been dealt with in Cases A and B.  

\medskip

{\bf There are no more cases to consider than (1) and (2):}  The cases that have not been dealt with before are (a) $\ell_m=0$, (b) $t = \infty$, (c) $\ell_m > \ell_{m-1}$ but not already covered Case A, and (d) $\ell_m < \ell_t$ but not already covered by Case B.

All cases with $\ell_m=0$ are still there. If $t=\infty$, then $[1,m-1]$ is the rightmost lower section in this case, and there exists a 
$k$ with $0 \le k \le m-1$ such that $1,2,\ldots,k \rightarrow m$ and $k+1,\ldots,m-1 \leftarrow m$ because $[1,m-1]$ is normal. The fact that the restriction to 
$\{1,2,\ldots,m\}$ is in $\mathcal{R}_m$ implies $m \rightarrow m$ and we can assume $k<m-1$ because otherwise $\ell_m=0$ and that is already covered. This is then covered by (2). 

If $t \not= \infty$ and $\ell_m > \ell_{m-1}$, we still need to consider the case $\ell_m-\ell_{m-1}>q-k$, because it has not been dealt with in Case A. We will show that this case can actually not happen. Let $[1,q]$ be the lower section after $m-1$ and, as usual, $0 \le k \le q$ such that 
$1,2,\ldots,k \rightarrow m$, $k+1,\ldots,q \leftarrow m$. Now $\ell_{m-1}=m-1-q$, so that $\ell_m-\ell_{m-1} > q-k$ is equivalent to $\ell_m > m-1-k$ and therefore $\ell_m \ge m-k$. As $1,2,\ldots,k \rightarrow m$, this implies $k+1,k+2,\ldots,m \leftarrow m$ (because these are $m-k$ edges) and $O|_{[m]} \in \mathcal{P}_m$, a contradiction.

If $t \not= \infty$ and $0 \not= \ell_m < \ell_t$, but the case is not covered by Case B.  Let 
$[p,q]$ be the lower section that appears in the linear order just before $t$, and let $p,p+1,\ldots,k \rightarrow m$ and 
$k+1,k+2,\ldots,q \leftarrow m$. Since $t$ is leftmost, $[p,q]$ is also the rightmost lower section and $\ell_t=q-p+2$ and $t=q+1$.
We can assume $\ell_t-\ell_m > k-p+1$ (because otherwise we are in Case B), so therefore $(q-p+2)-\ell_m > k-p+1$, which implies   
$q-k+1>\ell_m$, but since $k+1,k+2,\ldots,q \leftarrow m$ we have $\ell_m=q-k$, so that the left-pointing edges from $m$ hit 
precisely $k+1,\ldots,q$. We have $k<q$ since $\ell_m>0$. This is covered by (2). 

\bigskip

{\bf Now we show how (1) and (2) are turned into one another.} 

\medskip

Suppose we are in (1). Since $\ell_m=0$, we have $t \not= \infty$ because otherwise $O|_{[m]}$ has only right-pointing edges and would be contained in $\mathcal{P}_m$.
Let $[p,q]$ be the lower section that precedes $t$ (so that $t=q+1$), which is clearly the rightmost lower section.
The linear order of the vertices in $[m-1]$ starts as $q, q-1,\ldots,p,q+1$ and we change this to 
$q+1,q,q-1,\ldots,p$ with $q+1 \rightarrow q+1$. This replaces $Y_{q+1}^{q-p+2}$  with $X_p X_{p+1} \cdots X_{q+1}$. Moreover, we change 
$p,p+1,\ldots,q+1 \rightarrow m$ to $p,p+1,\ldots,q+1 \leftarrow m$, which replaces $X_p X_{p+1} \cdots X_{q+1}$ with 
$Y_m^{q-p+2}$. Summarizing, one weight is obtained from the other by applying the transposition $(q+1,m)$ when restricting to $[m]$. We exchange the neighbourhood of $q+1$ and $m$ in $\{m+1,m+2,\ldots,n\}$.

\medskip

Suppose we are in (2). Then the linear order of the vertices in $[m-1]$ starts as $q, q-1, \ldots, k+1$ and we change 
this to $q-1, q-2, \ldots,k+1, q$. This replaces $X_{k+1} X_{k+2} \ldots X_q$ with $Y_q^{q-k}$. We also change 
$k+1,k+2,\ldots,q \leftarrow m$ to $k+1,k+2,\ldots,q \rightarrow m$, so that $Y_m^{q-k}$ is replaced by  
$X_{k+1} X_{k+2} \ldots X_q$. So one weight is obtained from the other by applying the transposition $(q,m)$ when restricting to $[m]$. We exchange the neighbourhood of $q$ and $m$ in $\{m+1,m+2,\ldots,n\}$.
\end{proof}

%%%%%%%%%%%%%%%%%%%%%%%%%%%%%%%%%%%%%%%%%%%%%%%%%%%%%%%%%%%%%%%%%%%%%%%%%%%%%%%%%%%%%%%%%%%%%%%%%%%%%%%%%%%%
%%%%%%%%%%%%%%%%%%%%%%%%%%%%%%%%%%%%%%%%%%%%%%%%%%%%%%%%%%%%%%%%%%%%%%%%%%%%%%%%%%%%%%%%%%%%%%%%%%%%%%%%%%%%
%                                            5) Schur expansion
%%%%%%%%%%%%%%%%%%%%%%%%%%%%%%%%%%%%%%%%%%%%%%%%%%%%%%%%%%%%%%%%%%%%%%%%%%%%%%%%%%%%%%%%%%%%%%%%%%%%%%%%%%%%
%%%%%%%%%%%%%%%%%%%%%%%%%%%%%%%%%%%%%%%%%%%%%%%%%%%%%%%%%%%%%%%%%%%%%%%%%%%%%%%%%%%%%%%%%%%%%%%%%%%%%%%%%%%%

\section{The Schur expansion of the symmetric generating function}
\label{sec: proof of mt 1}

In order to prove Theorem~\ref{thm: main thm 1}, we first derive an explicit expansion of the symmetric generating function into Schur polynomials. Second, we prove that the coefficients of each Schur polynomial satisfy the same recursion as the right hand side of \eqref{eq: main thm 1}.
Let $\asym$ denote the antisymmetrizer, i.e.,
\[
\asym_{\x}  f(\mathbf{x}) = \sum\limits_{\sigma \in {\mathcal S}_n} \sgn(\sigma) \allowbreak \cdot f(x_{\sigma(1)},\ldots,x_{\sigma(n)})
.
\] 
We can rewrite the classical bialternant formula for Schur polynomials using the antisymmetriser and obtain for the operator formula in \eqref{eq: sym gen fct via Operatoren}
\begin{multline}
 \left.\prod_{1 \leq i < j \leq n} \left(u E_{\lambda_i}  + w E_{\lambda_i}E_{\lambda_j}^{-1} +  v E_{\lambda_j}^{-1}  \right) s_{(\lambda_n,\ldots,\lambda_1)}(\x)\right|_{\lambda_i=i-1} \\
=  \left. \prod_{1 \leq i < j \leq n} \left(u E_{\lambda_i}  + w E_{\lambda_i}E_{\lambda_j}^{-1} +  v E_{\lambda_j}^{-1}  \right) \frac{\asym_\x \left(\prod\limits_{i=1}^n x_i^{\lambda_i+i-1}\right)}{\prod\limits_{1 \leq i < j \leq n}(x_j-x_i)} \right|_{\lambda_i=i-1}\\
=   \frac{\asym_\x \left[
\prod\limits_{1 \leq i < j \leq n} \left(u x_i  + w x_i x_j^{-1} +  v x_j^{-1}  \right) \prod\limits_{i=1}^n x_i^{2(i-1)} \right]}{\prod\limits_{1 \leq i < j \leq n}(x_j-x_i)}, \label{eq: sym gen fct via antisym}
\end{multline}
where we used that $E_{\lambda_i}$ applied to $x_i^{\lambda_i}$ is equal to multiplication by $x_i$. By multiplying the $(i,j)$-th factor in the product with $x_i^{-1} x_j$ and some further manipulation, we arrive at
\begin{multline*}
\prod\limits_{i=1}^n \left( \frac{ x_i^{n-1}}{\frac{ v}{x_i} + w + u x_i} \right)
   \frac{ \asym_{\x} 
\left[  \prod\limits_{1 \le i \le  j \le n} \left(\frac{ v}{x_i} + w + u x_j\right) \right]}{ \prod\limits_{1 \le i < j \le n} (x_j - x_i)}
=\\
\prod\limits_{i=1}^n \left( \frac{ x_i^{n-1}}{\frac{ v}{x_i} + w + u x_i} \right)
   \frac{ \asym_{\x} 
\left[  \prod\limits_{1 \le i \le  j \le n} \left(\frac{ v}{x_j} + w + u x_i\right) \right]}{ \prod\limits_{1 \le i < j \le n} (x_i - x_j)},
\end{multline*}
where we replaced $x_i$ by $x_{n+1-i}$ for all $i$ in both the numerator and denominator.
We apply Lemma \ref{lem: general} for $X_i=-w-u x_i$ and $Y_i=v x_i^{-1}$, 
and obtain
\begin{equation}
\label{eq: sym gen fct via det}
\frac{\det_{1 \leq i,j \leq n}\left( x_i^{n-j}p_j(x_i) \right)}{\prod_{1 \leq i < j \leq n}(x_i-x_j)},
\end{equation}
where 
\[
p_j(x) = x^{j-1}\frac{v^j x^{-j}-\left( -w-u x\right)^j}{(v x^{-1}+w+u x)}
= \sum\limits_{k=0}^{j-1}x^k(-w-u x)^k  v^{j-k-1}.
\]
 To emphasise the general principle used to express the determinantal expression in \eqref{eq: sym gen fct via det} as a sum of Schur polynomials, we consider $q_j(x)$ to be a family of polynomials $q_j(x):= \sum_{k\geq 0} a_{j,k} x^k$.
Using the linearity of the determinant in the columns, we have
\begin{equation}
\label{eq: det to Schur I}
\frac{\det\limits_{1 \leq i,j \leq n}\left( x_i^{n-j}q_j(x_i) \right)}{\prod\limits_{1 \leq i< j \leq n}(x_i-x_j)}
%=\frac{\det\limits_{1 \leq i,j \leq n}\left( x_i^{n-j}\sum\limits_{k \geq 0} a_{j,k}x_i^k \right)}{\prod\limits_{1 \leq i< j \leq n}(x_i-x_j)}\\
%= 
%\sum\limits_{k_1,\ldots,k_n \geq 0} \prod_{j=1}^n a_{j,k_j}
%\frac{\det\limits_{1 \leq i,j \leq n}\left( x_i^{n-j+k_j} \right)}{\prod\limits_{1 \leq i< j \leq n}(x_i-x_j)}
= \sum\limits_{k_1,\ldots,k_n \geq 0} \left( \prod_{j=1}^n a_{j,k_j} \right) s_{(k_1,\ldots, k_n)}(\mathbf{x}),
\end{equation}
where we used the well known extension of Schur polynomials to arbitrary sequences $L=(L_1,\ldots,L_n)$ of non-negative integers via
\[
s_{L}(\mathbf{x}) := \frac{\det\limits_{1 \leq i,j \leq n}\left( x_i^{L_j+n-j}\right)}{\prod\limits_{1 \leq i < j \leq n}(x_i-x_j)}.
\]
It can be checked that the generalised Schur polynomial $s_L(\mathbf{x})$ is either equal to $0$ or $s_L(\mathbf{x})= \sgn(\sigma) s_\lambda(\mathbf{x})$ where $\lambda=(\lambda_1,\ldots, \lambda_n)$ is a partition 
and $\sigma \in S_n$ is a permutation such that $L_j=\lambda_{\sigma(j)}+j-\sigma(j)$ for all $1 \leq j \leq n$. It follows that \eqref{eq: det to Schur I} is equal to
\begin{equation}
\label{eq: det to Schur II}
%\frac{\det\limits_{1 \leq i,j \leq n}\left( x_i^{n-j}p_j(x_i) \right)}{\prod\limits_{1 \leq i< j \leq n}(x_i-x_j)}
 \sum_{\lambda} s_\lambda(\mathbf{x}) \left( \sum_{\sigma \in S_n} \sgn(\sigma) \prod_{j=1}^n a_{j,\lambda_{\sigma(j)}+j-\sigma(j)}  \right)
= \sum_{\lambda} s_\lambda(\mathbf{x}) \det\limits_{1 \leq i,j \leq n} \left( a_{j,\lambda_i+j-i} \right),
\end{equation}
where the sum is over all partitions $\lambda$.
By applying \eqref{eq: det to Schur II} to the family of polynomials 
$$
p_{j}(x) 
=\sum\limits_{k=0}^{j-1}x^k(-w-u x)^k  v^{j-k-1}
= \sum\limits_{k=0}^{j-1} \sum\limits_{l\geq 0}
(-1)^{k}\binom{k}{l} x^{k+l}   u^lv^{j-k-1} w^{k-l}
,
$$
we obtain
\begin{multline*}
\sum_{A \in \ASM_n}\omega_A(u,v,w;\x) =
\sum_\lambda s_\lambda(\mathbf{x}) \det\limits_{1 \leq i,j \leq n} \left( 
\sum\limits_{k=0}^{j-1}  \sum\limits_{l \geq 0 \atop k+l = \lambda_i+j-i}
(-1)^{k}\binom{k}{l}   u^lv^{j-k-1} w^{k-l}
\right)\\
= \sum_\lambda s_\lambda(\mathbf{x})  \det_{1 \leq i,j \leq n} \left(
\sum_{k=0}^{j-1} (-1)^{k}\binom{k}{\lambda_i+j-i-k} u^{\lambda_i+j-i-k} v^{j-k-1} w^{2k+i-\lambda_i-j}
\right).
\end{multline*}

We denote by $m_{i,j}(\lambda_i)$ the $(i,j)$-th entry of the matrix in the above determinant. An entry $m_{i,1}(\lambda_i)= \binom{0}{\lambda_i+1-i}u^{\lambda_i+1-i} w^{i-\lambda_i-1} $ in the first column is $1$ iff $\lambda_i=i-1$ and $0$ otherwise. Let $l$ be the side length of the Durfee square of $\lambda$. The only possible part of $\lambda$ satisfying $\lambda_i=i-1$ is the $(l+1)$-st. Hence we assume for the rest of the proof $\lambda_{l+1}=l$. By expanding the determinant along the first column, we obtain
\[
\det_{1 \leq i,j \leq n} \left(m_{i,j}(\lambda_i) \right) = (-1)^{l+2}\det_{1 \leq i,j \leq n-1} \left(m_{i,j}^\prime \right),
\]
where $(m_{i,j}^\prime)_{1,\leq,i,j \leq n-1}$ denotes the matrix obtained by deleting the first column and the $(l+1)$-st row of $(m_{i,j}(\lambda_i))_{1\leq i,j \leq n}$. For $1 \leq i \leq l$, in which case we have $\lambda_{i}\geq i$, we can rewrite $m_{i,j}^\prime$ as
\begin{multline*}
m_{i,j}^\prime =
\sum_{k=0}^{j} (-1)^k \binom{k}{\lambda_i+(j+1)-i-k}  u^{\lambda_i+(j+1)-i-k} v^{(j+1)-k-1} w^{2k+i-\lambda_i-(j+1)} \\
= \sum_{k=0}^{j-1} (-1)^{k+1}    u^{\lambda_i+j-i-k}v^{j-k-1} w^{2k+1+i-\lambda_i-j}
 \left( \binom{k}{\lambda_i+j-i-k}+ \binom{k}{\lambda_i+j-i-k-1} \right) \\
 = -w \, m_{i,j}(\lambda_i) - u \, m_{i,j}(\lambda_i-1).
\end{multline*}
For $i>l$ on the other hand, i.e., $\lambda_{i+1}<i$, we can express $m_{i,j}^\prime$ analogously as
\begin{multline*}
m_{i,j}^\prime = \sum_{k=0}^{j} (-1)^{k}\binom{k}{\lambda_{i+1}+(j+1)-(i+1)-k} \\
\times  u^{\lambda_{i+1}+(j+1)-(i+1)-k} v^{(j+1)-k-1} w^{2k+(i+1)-\lambda_{i+1}-(j+1)}
=  v m_{i,j}(\lambda_{i+1}),
\end{multline*}
where the sum has been extended, which is allowed since $\binom{j}{\lambda_{i+1}-i}=0$.
Summarising, we denote by $c_{n,\lambda}$ the coefficient of $s_\lambda(\mathbf{x})$ in the symmetric generating function $\sum_{A \in \ASM_n}\omega_A(u,v,w;\x)$. Then
\begin{multline*}
c_{n,\lambda} = (-1)^{l}\det_{1 \leq i,j \leq n-1} \left(m_{i,j}^\prime \right) 
 = (-1)^{l}\det_{1 \leq i,j \leq n-1} \left( 
\begin{cases}
-w m_{i,j}(\lambda_i)- u m_{i,j}(\lambda_i-1), \quad & i \leq l\\
 v m_{i,j}(\lambda_{i+1}), & i>l
\end{cases} \right) \\
= \sum_{(f_1,\ldots,f_l) \in \{0,1\}^l} u^{\sum_{i=1}^l f_i}  v^{n-1-l} w^{l-\sum_{i=1}^l f_i}
  c_{n-1,(\lambda_1 -f_1,\ldots,\lambda_l-f_l,\lambda_{l+2},\ldots,\lambda_n)},
\end{multline*}
with $c_{n-1,(\lambda_1 -f_1,\ldots,\lambda_l-f_l,\lambda_{l+2},\ldots,\lambda_n)}=0$ if $(\lambda_1 -f_1,\ldots,\lambda_l-f_l,\lambda_{l+2},\ldots,\lambda_n)$ is not a partition, where the equality follows from the linearity of the determinant in the rows and choosing $f_i=0$ iff we select the first term in row $i$.
Using Frobenius notation for $\lambda=(a_1,\ldots,a_l|b_1,\ldots,b_l)$, the above recursion can be rewritten as
\begin{multline*}
c_{n,(a_1,\ldots,a_l|b_1,\ldots,b_l)} 
= \sum_{(f_1,\ldots,f_l) \in \{0,1\}^l} u ^{\sum_{i=1}^l f_i} v^{n-1-l} w^{l-\sum_{i=1}^l f_i}  c_{n-1,(a_1-f_1,\ldots,a_l-f_l|b_1-1,\ldots,b_l-1)},
\end{multline*}
where $c_{n-1,(a_1,\ldots,a_{l-1},-1|b_1,\ldots,b_{l-1},0)}$ is defined as $c_{n-1,(a_1,\ldots,a_{l-1}|b_1,\ldots,b_{l-1})}$ .
\bigskip

Denote by $d_{n,\lambda}$ the coefficient of $s_\lambda(\x)$ in $\A_{n,1}(1, u,v,w;\x)$. For $\lambda=(a_1,\ldots,a_l|b_1,\ldots,b_l)$, Proposition \ref{prop: TSSPP refinement} implies 
\begin{multline*}
d_{n,(a_1,\ldots,a_l|b_1,\ldots,b_l)} 
= u^{\sum_{i=1}^l (a_i+1)} v^{\binom{n}{2}-\sum_{i=1}^l b_i} w^{\sum_{i=1}^l(b_i-1-a_i)} 
 \det_{1 \leq i,j \leq l} \left( \binom{b_j-1}{a_i} \right)\\
= u^{\sum_{i=1}^l (a_i+1)} v^{\binom{n}{2}-\sum_{i=1}^l b_i} w^{\sum_{i=1}^l(b_i-1-a_i)} 
 \det_{1 \leq i,j \leq l} \left( \binom{b_j-2}{a_i}+\binom{b_j-2}{a_i-1} \right)\\
 =  \sum_{(f_1,\ldots,f_l) \in \{0,1\}^l}  u^{\sum_{i=1}^l f_i} v^{n-1-l} w^{l-\sum_{i=1}^l f_i} 
 d_{n-1,(a_1-f_1,\ldots,a_l-f_l|b_1-1,\ldots,b_l-1)},
\end{multline*}
where we used the linearity of the determinant in the last step. The assertion follows by induction on $n$ since both $c_{n,\lambda}$ and $d_{n,\lambda}$ satisfy the same recursion and the induction base can be checked easily. This proves Theorem~\ref{thm: main thm 1}.

%%%%%%%%%%%%%%%%%%%%%%%%%%%%%%%%%%%%%%%%%%%%%%%%%%%%%%%%%%%%%%%%%%%%%%%%%%%%%%%%%%%%%%%%%%%%%%%%%%%%%%%%%%%%
%%%%%%%%%%%%%%%%%%%%%%%%%%%%%%%%%%%%%%%%%%%%%%%%%%%%%%%%%%%%%%%%%%%%%%%%%%%%%%%%%%%%%%%%%%%%%%%%%%%%%%%%%%%%
%                                            6) CSSPPs and sym fcts
%%%%%%%%%%%%%%%%%%%%%%%%%%%%%%%%%%%%%%%%%%%%%%%%%%%%%%%%%%%%%%%%%%%%%%%%%%%%%%%%%%%%%%%%%%%%%%%%%%%%%%%%%%%%
%%%%%%%%%%%%%%%%%%%%%%%%%%%%%%%%%%%%%%%%%%%%%%%%%%%%%%%%%%%%%%%%%%%%%%%%%%%%%%%%%%%%%%%%%%%%%%%%%%%%%%%%%%%%

\section{$\A_{n,k}$ and column strict shifted plane partitions}
\label{sec: CSSPPs}

Recall that a strict partition is a sequence $\lambda=(\lambda_1,\ldots,\lambda_n)$ of strictly decreasing positive integers. The {shifted Young diagram} of shape $\lambda$ has $\lambda_i$ cells in row $i$ and each row is indented by one cell to the right with respect to the previous row. The shifted Young diagram of the strict partition $(6,5,2)$ is as follows.
$$
\tiny
\ydiagram{0+6,1+5,2+2}
$$
A {column strict shifted plane partition}  (CSSPP) is a filling of a shifted Young diagram with positive integers such that rows decrease weakly and columns decrease strictly. 
Let $k$ be an integer, then a CSSPP is said to be of class $k$ if the first part of each row exceeds the length of its row by precisely $k$.
The following is a CSSPP of class $2$.
\[
\begin{array}{cc cc cc}
 8 & 8 & 7 & 6 & \cyan{3} & \cyan{2}\\
   & 7 & 5 & \cyan{2} & \cyan{1} & \cyan{1}\\
   &   & 4 & \cyan{1}
\end{array}
\]
For a CSSPP $\pi$ of class $k$, we define $\rho(\pi)$ as the number of rows of $\pi$ and $\mu(\pi)$ as the number of entries $\pi_{i,j} \leq k+j-i$. In the above example, the two statistics are $\rho(\pi)=3$ and $\mu(\pi)=6$, where the entries contributing to $\mu(\pi)$ are coloured blue. We define the function $\CSSPP_{n,k}(r,t)$ as the generating function
\[
\CSSPP_{n,k}(r,t)= \sum_{\pi}r^{\rho(\pi)}t^{\mu(\pi)},
\]
where the sum is over all CSSPP $\pi$ of class $k$ whose first row has at most $n$ entries. Using a lattice path description for CSSPPs and the Lindstr\"om-Gessel-Viennot theorem, we obtain the following determinantal formula for $\CSSPP_{n,k}(r,t)$. A detailed proof can be found in \cite[Lemma 5.1]{Aigner21}.

\begin{prop}
\label{prop: CSSPP gen fct}
Let $n$ be a positive integer and $k$ a non-negative integer. Then 
\[
\CSSPP_{n,k}(r,t)= \det_{0 \leq i,j \leq n-1} \left(\delta_{i,j} + r\sum_{l \geq 0}\binom{i}{l}\binom{j+k}{l+k}t^{j-l} \right).
\]
\end{prop}

It is crucial for the proof of Theorem \ref{thm: main thm 2} to express $\A_{n,k}(r,1,1,t;\x)$ as a determinant. The next lemma gives a determinantal expression for the more general $\A_{n,k}(r,u,v,w;\x)$.

\begin{lem}
\label{lem: Ank via det}
Let $n \geq 2$ be an integer, then
\begin{equation}
\label{eq: Ank via det}
\A_{n,k}(r, u,v,w;\x) = \det_{0 \leq i,j \leq n-2} \left((-1)^{j-i}  v^{j+1} \binom{i}{j}+r  u^{i+1} w^{j-i}  s_{(i|j+k)}(\x)\right).
\end{equation}
\end{lem}
\begin{proof}
We denote by $[0,n] = \{0,1,\ldots,n\}$. Expanding the determinant by the Leibniz formula yields
\begin{multline}
\label{eq: Ank via det proof 1}
\sum_{\sigma \in S_{[0,n-2]}} \sgn(\sigma) \prod_{i=0}^{n-2} \left((-1)^{\sigma(i)-i}  v^{\sigma(i)+1} \binom{i}{\sigma(i)}+r u^{i+1} w^{\sigma(i)-i} 
s_{(i|\sigma(i)+k)}(\x) \right)\\
= \sum_{A \subseteq [0,n-2]} \sum_{\sigma \in S_{[0,n-2]}} \sgn(\sigma) \prod_{i \in A}\left( r u^{i+1} w^{\sigma(i)-i}  s_{(i|\sigma(i)+k)}(\x) \right)\\ \times
\prod_{i \in [0,n-2]\setminus A} \left((-1)^{\sigma(i)-i}  v^{\sigma(i)+1} \binom{i}{\sigma(i)}\right).
\end{multline}
For $A =\{a_1,\ldots,a_l\} \subseteq [0,n-2]$ with $a_1> \ldots >a_l$ denote\footnote{The notation is used in a similar way as in \eqref{eq: Frobenius coef of complement}.} by $B^c=\{b_1^c,\ldots,b^c_{n-1-l}\} = [0,n-2] \setminus A$ the complement of $A$ where $b^c_1 > \ldots > b^c_{n-1-l}$. For a permutation $\sigma \in S_{[0,n-2]}$ denote by $b_1 > \ldots > b_l$ the elements of the image of $A$ and by $a^c_1 > \ldots > a^c_{n-1-l}$ the elements of the image of $B^c$. Define $\pi \in S_l$ and $\tau \in S_{n-l-1}$ via
$\sigma(a_i) = b_{\pi(i)}$ and $\sigma(b^c_i) =a^c_{\tau(i)}$. It is not complicated to see that the sign of $\sigma$ is given by 
\[
\sgn(\sigma)=\sgn(\pi)\sgn(\tau) \prod_{i \in [0,n-2]\setminus A}(-1)^{\sigma(i)-i}.
\]
For a given set $A \subseteq [n-1]$, the permutation $\sigma$ is uniquely determined by $\{b_1,\ldots,b_l\}$ and the permutations $\pi,\tau$. Note that $\la=(a_1,\ldots,a_l|b_1,\ldots,b_l)$ yields a partition inside $(n-1)^{n-1}$. Hence we can rewrite \eqref{eq: Ank via det proof 1} as 
\begin{multline*}
\sum_{\lambda=(a_1,\ldots,a_l|b_1,\ldots,b_l)\subseteq (n-1)^{n-1}} \,\, \sum_{\pi \in S_l} \,\, \sum_{\tau \in S_{n-l-1}} \sgn(\pi) \sgn(\tau) \\
\times \prod_{i=1}^l r u^{a_i+1} w^{b_{\pi(i)}-a_i}  s_{(a_i|b_{\pi(i)}+k)}(\x)
\prod_{i=1}^{n-l-1}  v^{a^c_{\tau(i)}+1} \binom{b^c_i}{a^c_{\tau(i)}} \\
= \sum_{\lambda=(a_1,\ldots,a_l|b_1,\ldots,b_l)\subseteq (n-1)^{n-1}}
r^l u^{\sum_{i=1}^l (a_i+1)} v^{\binom{n}{2}-\sum_{i=1}^l (b_i+1)} w^{\sum_{i=1}^l (b_i-a_i)}  \\
\times \det_{1 \leq i,j \leq l}\left( s_{(a_i|b_j+k)}(\x)\right)
\det_{1 \leq i,j \leq n-l-1}\left( \binom{b^c_i}{a^c_j} \right),
\end{multline*}
where we used $\sum_{i=1}^{n-l-1}(a_i^c+1) +\sum_{i=1}^l (b_i+1) = \binom{n}{2}$ in the last step.
Using \eqref{eq: det of complement} and the Giambelli identity which states 
\[
s_{(a_1,\ldots,a_l|b_1+k,\ldots,b_l+k)}(\x)= \det_{1 \leq i,j \leq l}\left( s_{(a_i|b_j+k)}(\x)\right),
\]
we can rewrite the above as
\begin{multline*}
\sum_{\lambda=(a_1,\ldots,a_l|b_1,\ldots,b_l)\subseteq (n-1)^{n-1}}
r^l u^{\sum_{i=1}^l (a_i+1)}  v^{\binom{n}{2}-\sum_{i=1}^l (b_i+1)} w^{\sum_{i=1}^l (b_i-a_i)} \\
\times s_{(a_1,\ldots,a_l|b_1+k,\ldots,b_l+k)}(\x)
\det_{1 \leq i,j \leq l}\left( \binom{b_i}{a_j} \right),
\end{multline*}
which is equal to $\A_{n,k}(r, u,v,w;\x)$ by Proposition \ref{prop: TSSPP refinement}.
\end{proof}

%%%%%%%%%%%%%%%%%%%%%%%%%%%%%%%%%%%%%%%%%%%%%%%%%%%%%%%%%%%%%%%%%%%%%%%%%%%%%%%%%%%%%%%%%%%%%%%%%%%%%%%%%%%%
%%%%%%%%%%%%%%%%%%%%%%%%%%%%%%%%%%%%%%%%%%%%%%%%%%%%%%%%%%%%%%%%%%%%%%%%%%%%%%%%%%%%%%%%%%%%%%%%%%%%%%%%%%%%
%                                            7) Proofs for CSSPPs
%%%%%%%%%%%%%%%%%%%%%%%%%%%%%%%%%%%%%%%%%%%%%%%%%%%%%%%%%%%%%%%%%%%%%%%%%%%%%%%%%%%%%%%%%%%%%%%%%%%%%%%%%%%%
%%%%%%%%%%%%%%%%%%%%%%%%%%%%%%%%%%%%%%%%%%%%%%%%%%%%%%%%%%%%%%%%%%%%%%%%%%%%%%%%%%%%%%%%%%%%%%%%%%%%%%%%%%%%

\section{Proof of Theorem \ref{thm: main thm 2}}
\label{sec: proof of mt 2}

Using the hook-content formula, we can express the evaluation of the Schur polynomial $s_{(a|b)}(\x)$ at $x_i=1$ as
\[
\left.s_{(a|b)}(x_1,\ldots,x_{n+k-1})\right|_{x_i=1}=\binom{n+k-1+a}{a+b+1}\binom{a+b}{a}.
\]
Together with Lemma \ref{lem: Ank via det} we obtain
\begin{equation}
\label{eq: Ank x=1}
\A_{n,k}(r, u, v, w;\mathbf{1}) = \det_{0 \leq i,j \leq n-2}\left((-1)^{j-i}  v^{j+1} \binom{i}{j}+ r u^{i+1} w^{j-i}  \binom{n+k+i-1}{i+j+k+1}\binom{i+j+k}{i} \right).
\end{equation}

We also need the following transformation identity for a binomial sum for the proof of Theorem \ref{thm: main thm 2}.

\begin{lem}
\label{lem: binomial identity}
Let $a,b,c$ be non-negative integers with $a,c \leq b$ and $x$ a variable, then
\begin{equation}
\label{eq: binomial identity}
\sum_{l=0}^b \binom{l}{c}\binom{x+l}{l-a} = \sum_{s=0}^c \binom{x+b+s+1}{b-a} \binom{x+a+s}{s} \binom{x+c}{c-s} (-1)^{c-s}.
\end{equation}
\end{lem}
\begin{proof}
Using hypergeometric notation, we can rewrite the left-hand side as
\[
\pFq{3}{2}{c-b,1,a-b}{-b,-b-x}{1}\binom{b}{c}\binom{x+b}{x+a}.
\]
We apply the ${}_3F_2$-series transformation \cite[(3.1.1)]{GasperRahman90}
\[
\pFq{3}{2}{a,b,-n}{d,e}{1} = \pFq{3}{2}{d-a,b,-n}{d,1+b-e-n}{1}\frac{(e-b)_n}{(e)_n},
\]
and obtain
\[
\pFq{3}{2}{-c,1,a-b}{-b,2+a+x}{1}\binom{b}{c}\binom{x+b+1}{x+a+1}.
\]
By further applying the terminating form of the ${}_3F_2$-series transformation \cite[Ex. 7, p. 98]{Bailey35}
\[
 \pFq{3}{2}{-n,a,b}{d,e}{1} = \pFq{3}{2}{-n,e-a,e-b}{e,d+e-a-b}{1} \frac{(d+e-a-b)_n}{(d)_n},
\]
we have
\[
\pFq{3}{2}{-c,1+a+x,2+b+x}{2+a+x,1+x}{1}\binom{x+c}{c}\binom{x+b+1}{x+a+1}(-1)^c,
\]
which is the right-hand side of \eqref{eq: binomial identity} expressed as a hypergeometric series.
\end{proof}

\subsection{Proof of \eqref{eq: Ank CSSPP 1}}
\begin{proof}
The assertion follows from the matrix identity
\begin{multline*}
\left(\binom{i}{j}t^{j+1-n}(t+1)^{i-j}(t+2)^{n-i-1} \right)_{0 \leq i,j \leq n-1}\\ \cdot
\left((-1)^{i+j}\binom{i}{j}+rt^{j-i}\binom{n+i}{n-j-1}\binom{i+j}{i} \right)_{0 \leq i,j \leq n-1}\\
=
\left(\delta_{i,j} + r\sum_{l=0}^{n-1} \binom{i}{l}\binom{j}{l}(t+2)^{j-l} \right)_{0 \leq i,j \leq n-1} \cdot
\left(\binom{i}{j}t^{j+1-n}(t+2)^{n-i-1} \right)_{0 \leq i,j \leq n-1}.
\end{multline*}
Indeed, the first and fourth matrices are lower triangular matrices and their corresponding determinants are both equal to $\prod_{i=0}^{n-1} t^{i+1-n}(t+2)^{n-i-1}$. The determinant of the second matrix is equal to $\A_{n+1,0}(r,1,1,t;\mathbf{1})$ by \eqref{eq: Ank x=1} and the determinant of the third matrix is equal to $\CSSPP_{n,0}(r,t+2)$ by Proposition \ref{prop: CSSPP gen fct}. Hence the assertion follows by taking determinants on both sides of the matrix identity. \bigskip

To show the above matrix identity, we first use matrix multiplication and obtain for the $(i,j)$-th entry
\begin{multline}
\label{eq: matrix id I 2}
\sum_{l=0}^{n-1} \binom{i}{l}t^{l+1-n}(t+1)^{i-l}(t+2)^{n-i-1} \left((-1)^{l+j}\binom{l}{j} +r t^{j-l}\binom{n+l}{n-j-1}\binom{l+j}{l} \right) \\
=\sum_{s=0}^{n-1}\left(\delta_{i,s} +r \sum_{l=0}^{n-1}\binom{i}{l}\binom{s}{l}(t+2)^{s-l} \right)
\binom{s}{j}t^{j+1-n}(t+2)^{n-s-1}.
\end{multline}
The sum over terms not involving the variable $r$ on the left-hand side of \eqref{eq: matrix id I 2} is
\[
\sum_{l=0}^{n-1}(-1)^{l+j}\binom{i}{l}\binom{l}{j}t^{l+1-n}(t+1)^{i-l}(t+2)^{n-i-1}.
\]
By using $\binom{i}{l}\binom{l}{j}=\binom{i}{j}\binom{i-j}{i-l}$ and the binomial theorem, we can rewrite the above sum as
\[
\binom{i}{j} t^{j+1-n}(t+2)^{n-i-1}\sum_{l=0}^{n-1} \binom{i-j}{i-l} (-t)^{(i-j)-(i-l)}(t+1)^{i-l} =\binom{i}{j} t^{j+1-n}(t+2)^{n-i-1},
\]
which is equal to the $r$-free term on the right-hand side of \eqref{eq: matrix id I 2}.
The sum over the terms of the right-hand side of \eqref{eq: matrix id I 2} involving the variable $r$ is equal to
\[
r t^{j+1-n}\sum_{l=0}^{n-1} \binom{i}{l} (t+2)^{n-l-1} \sum_{s=0}^{n-1}\binom{s}{l}\binom{s}{j},
\]
where we interchanged the order of the summation. Using Lemma \ref{lem: binomial identity} for the sum over $s$ with $a=j$, $b=n-1$, $c=l$, $x=0$, we obtain
\[
r t^{j+1-n}\sum_{l=0}^{n-1} \binom{i}{l} (t+2)^{n-l-1} \sum_{s=0}^{n-1} \binom{n+s}{n-1-j}\binom{s+j}{s}\binom{l}{l-s}(-1)^{l-s},
\]
where the upper bound of the second sum can be changed to $n-1$, since the last binomial coefficient is $0$ for $l < s \leq n-1$. Interchanging the sums again and using the binomial theorem yields
\begin{multline*}
r t^{j+1-n} \sum_{s=0}^{n-1} \binom{n+s}{n-1-j}\binom{s+j}{s} \sum_{l=0}^{n-1} \binom{i}{l}\binom{l}{l-s} (t+2)^{n-l-1} (-1)^{l-s}\\
= r t^{j+1-n}(t+2)^{n-i-1}  \sum_{s=0}^{n-1} \binom{i}{s}\binom{n+s}{n-1-j}\binom{s+j}{s} (t+1)^{i-s},
\end{multline*}
which is equal to the terms of the left-hand side of \eqref{eq: matrix id I 2} involving the variable $r$.
\end{proof}

\subsection{Proof of \eqref{eq: Ank CSSPP 2}}
\begin{proof}
By factoring out $(-1)^{i+j}$ in the determinant expression of $\A_{n+1,k}(r,1,1,-1;\mathbf{1})$ in \eqref{eq: Ank x=1}, we obtain
\[
\A_{n+1,k}(r,1,1,-1;\mathbf{1})=\det_{0 \leq i,j \leq n-1} \left( \binom{i}{j} +r \binom{n+k+i}{n-j-1}\binom{i+j+k}{i} \right).
\]
Using the Chu-Vandermonde identity, the determinant for $\CSSPP_{n,2k}(r,1)$ in Proposition \ref{prop: CSSPP gen fct} simplifies to
\[
\CSSPP_{n,2k}(r,1) = \det_{0 \leq i,j \leq n-1}\left( \delta_{i,j} + r\binom{2k+i+j}{j} \right).
\]
We claim the following matrix identity
\begin{multline*}
\left(  \binom{j}{i} +r \binom{n+k+j}{n-i-1}\binom{i+j+k}{j}  \right)_{0 \leq i,j \le n-1}
\cdot \left( \binom{k+j-i-1}{j-i} \right)_{0 \leq i,j \le n-1}  \\
=
\left(\binom{k+j}{j-i} \right)_{0 \leq i,j \le n-1} \cdot
\left( \delta_{i,j} + r\binom{2k+i+j}{j}\right)_{0 \leq i,j \le n-1}.
\end{multline*}
Since the second and third matrices are upper triangular with determinant equal to $1$, the assertion \eqref{eq: Ank CSSPP 2} follows by taking the determinant of all matrices in the above identity.\bigskip

To prove the above matrix identity we use matrix multiplication and obtain for the $(i,j)$-th term
\begin{multline}
\label{eq: matrix id II 2}
\sum_{l=0}^{n-1} 
\left(  \binom{l}{i} +r \binom{n+k+l}{n-i-1}\binom{i+l+k}{l}  \right)\binom{k+j-l-1}{j-l}\\
= \sum_{l=0}^{n-1} \binom{k+l}{l-i}
\left( \delta_{l,j} + r\binom{2k+l+j}{j}\right).
\end{multline}
The Chu-Vandermonde identity implies 
\[
\sum_{l=0}^{n-1} \binom{k+j-l-1}{j-l}\binom{l}{i} = \binom{k+j}{j-i},
\]
which explains the terms of \eqref{eq: matrix id II 2} not involving the variable $r$.
By setting $l=L-(2k+j)$, the coefficient of $r$ on the right-hand side of \eqref{eq: matrix id II 2} is equal to
\[
\sum_{L=2k+j}^{2k+j+n-1}\binom{L-(j+k)}{L-(2k+j+i)}\binom{L}{j}.
\]
We can actually change the lower bound of the sum to $0$ since the first binomial coefficient is equal to $0$ for $0 \leq L < 2k+j$.
Using Lemma \ref{lem: binomial identity} for $a=2k+j+i$, $b=2k+j+n-1$, $c=j$, and $x=-(j+k)$ as well as $\binom{-k}{j-s}(-1)^{j-s}=\binom{j+k-s-1}{j-s}$ yields the coefficient of $r$ of the left-hand side of \eqref{eq: matrix id II 2}.
\end{proof}

\subsection{Proof of \eqref{eq: Ank CSSPP 3}}
\begin{proof}
We expand the determinant for $\A_{n+1,k}(r,1,1,t;\mathbf{1})$ in \eqref{eq: Ank x=1} by the Leibniz formula and obtain
\begin{multline*}
\A_{n+1,k}(r,1,1,t;\mathbf{1}) \\
= \sum_{\sigma \in S_{[0,n-1]}} \sgn(\sigma) \prod_{i=0}^{n-1}
\left( (-1)^{\sigma(i)-i}\binom{i}{\sigma(i)}+r t^{\sigma(i)-i}\binom{n+k+i}{n-\sigma(i)-1}\binom{i+\sigma(i)+k}{i}\right) \\
= \sum_{\sigma \in S_{[0,n-1]}} \sgn(\sigma) \sum_{I \subseteq [0,n-1]}
\prod_{i \in I} \left( r t^{\sigma(i)-i} \binom{n+k+i}{n-\sigma(i)-1}\binom{i+\sigma(i)+k}{i}\right)\\
\times \prod_{i \in [0,n-1]\setminus I} (-1)^{\sigma(i)-i}\binom{i}{\sigma(i)}.
\end{multline*}
Now note that the summand is $0$ unless $\sigma(i) \le i$ for all $i \in [0,n-1] \setminus I$, and, therefore, we restrict our sum to such $I$. The power of $t$ is $\sum_{i \in I} (\sigma(i)-i)$, which is non-negative as 
$\sum_{i \in [0,n-1]} (i-\sigma(i))=0$ and $\sum_{i \in [0,n-1] \setminus I} (\sigma(i)-i) \le 0$, and, therefore, we can now set $t=0$. However, after this specialisation, the summand is zero unless $\sum_{i \in I} (\sigma(i)-i)=0$, and, therefore, $\sum_{i \in [0,n-1] \setminus I} (\sigma(i)-i) =0$, which implies $\sigma(i)=i$ for all $i \in [0,n-1] \setminus I$.
%\begin{multline*}
%\A_{n+1,k}(r,1,1,t;\mathbf{1}) \\
%= \sum_{\sigma \in S_{[0,n-1]}} \sgn(\sigma) \prod_{i=0}^{n-1}
%\left( (-1)^{\sigma(i)-i}\binom{i}{\sigma(i)}+r t^{\sigma(i)-i}\binom{n+k+i}{n-\sigma(i)-1}\binom{i+\sigma(i)+k}{i}\right) \\
%= \sum_{\sigma \in S_{[0,n-1]}} \sgn(\sigma) \sum_{\substack{I \subseteq [0,n-1] \\ \{i: \sigma(i)>i\} \subseteq I }}
%\prod_{i \in I} \left( r t^{\sigma(i)-i} \binom{n+k+i}{n-\sigma(i)-1}\binom{i+\sigma(i)+k}{i}\right)\\
%\times \prod_{i \in [0,n-1]\setminus I} (-1)^{\sigma(i)-i}\binom{i}{\sigma(i)}.
%\end{multline*}
%For a given $\sigma \in S_{[0,n-1]}$ and $\{i: \sigma(i)>i\} \subseteq I \subseteq [0,n-1]$ the corresponding summand has a positive power of $t$ unless $I$ contains all $i$ with $\sigma(i) \neq i$, i.e., the set $[0,n-1]\setminus I$ is a subset of the fixed points of $\sigma$. 
Hence, for $t=0$, the above simplifies to
\begin{multline*}
\A_{n+1,k}(r,1,1,0;\mathbf{1}) \\
= \sum_{\sigma \in S_{[0,n-1]}} \sgn(\sigma) \sum_{\substack{I \subseteq [0,n-1] \\ \{i: \sigma(i) \neq i\} \subseteq I }}
\prod_{i \in I} \left( r \binom{n+k+i}{n-\sigma(i)-1}\binom{i+\sigma(i)+k}{i}\right) \\
= \sum_{\sigma \in S_{[0,n-1]}} \prod_{i=0}^{n-1} \left(\delta_{i,\sigma(i)}+r \binom{n+k+i}{n-\sigma(i)-1}\binom{i+\sigma(i)+k}{i}\right)\\
= \det_{0 \leq,i,j \leq n-1}\left( \delta_{i,j} + r \binom{n+k+i}{n-j-1}\binom{i+j+k}{i}\right).
\end{multline*}
Taking the determinant of the following matrix identity implies the  assertion \eqref{eq: Ank CSSPP 3}, since the second and third matrix are upper triangular with determinant equal to $1$ and the determinant of the fourth matrix is equal to $\CSSPP_{n,k}(r,2)$ by Proposition \ref{prop: CSSPP gen fct}.
\begin{multline*}
\left( \delta_{i,j} + r \binom{n+k+j}{n-i-1}\binom{i+j+k}{j} \right)_{0 \leq i,j \leq n-1} \cdot
\left( \binom{k+j}{j-i} \right)_{0 \leq i,j \leq n-1} \\ =
\left( \binom{k+j}{j-i} \right)_{0 \leq i,j \leq n-1} \cdot
\left( \delta_{i,j} +r \sum_{l\geq 0} \binom{i}{l}\binom{j+k}{l+k}2^{j-l} \right)_{0 \leq i,j \leq n-1} .
\end{multline*}
In order to prove the matrix identity, it suffices to show
\begin{equation}
\label{eq: matrix id III 2}
\sum_{s=0}^{n-1} \binom{n+k+s}{n-i-1}\binom{i+s+k}{s} \binom{k+j}{j-s}
 = \sum_{s,l=0}^{n-1} \binom{k+s}{s-i}\binom{s}{l}\binom{j+k}{l+k}2^{j-l}.
\end{equation}
We can rewrite the right-hand side of \eqref{eq: matrix id III 2} by applying Lemma \ref{lem: binomial identity} for the sum over $s$ with $a=i$, $b=n-1$, $c=l$ and $x=k$ and obtain
\begin{multline*}
\sum_{l=0}^{n-1} \binom{j+k}{l+k}2^{j-l} \sum_{s=0}^{l}  \binom{k+n+s}{n-i-1}\binom{k+i+s}{s}\binom{k+l}{l-s}(-1)^{l-s}\\
= \sum_{s=0}^{n-1}\binom{k+n+s}{n-i-1}\binom{k+i+s}{s}  \sum_{l=0}^{n-1}\binom{j+k}{l+k}\binom{k+l}{l-s}2^{j-l}(-1)^{l-s},
\end{multline*}
where we interchanged the sums and changed the upper bound of the sum over $s$ to $n-1$ which is allowed since $\binom{k+l}{l-s}=0$ for $s >l$.
By using $\binom{j+k}{l+k}\binom{k+l}{l-s}=\binom{j+k}{j-s}\binom{j-s}{l-s}$ together with the binomial theorem, we obtain the left-hand side of \eqref{eq: matrix id III 2}.
\end{proof}

\section*{Acknowledgements}
The authors want to thank Fran\c{c}ois Bergeron, Matja{\v z} Konvalinka, Philippe Nadeau and Vasu Tewari for helpful discussions. 

\bibliographystyle{abbrv}
\bibliography{LiteraturListe}

\end{document}